\documentclass[12pt, reqno]{amsart}
\usepackage{amsmath, amsthm, amscd, amsfonts, amssymb, graphicx, color}
\usepackage[bookmarksnumbered, colorlinks, plainpages]{hyperref}
\hypersetup{colorlinks=true,linkcolor=red, anchorcolor=green,
citecolor=cyan, urlcolor=red, filecolor=magenta, pdftoolbar=true}

\newtheorem{theorem}{Theorem}[section]
\newtheorem{lemma}[theorem]{Lemma}
\newtheorem{proposition}[theorem]{Proposition}
\newtheorem{corollary}[theorem]{Corollary}
\theoremstyle{definition}
\newtheorem{definition}[theorem]{Definition}
\newtheorem{example}[theorem]{Example}

\theoremstyle{remark}
\newtheorem{remark}[theorem]{Remark}
\numberwithin{equation}{section}

\begin{document}

\setcounter{page}{1}

\title[On an extension of operator transforms]
{On an extension of operator transforms}

\author[A. Zamani]
{Ali Zamani}

\address{Department of Mathematics, Farhangian University, Tehran, Iran}
\email{zamani.ali85@yahoo.com}

\subjclass[2010]{47A05, 47B49, 47A12, 47A30.}

\keywords{Polar decomposition, Duggal transform, mean transform, numerical radius, inequality.}
\begin{abstract}
We introduce the $\lambda$-mean transform $M_{\lambda}(T)$
of a Hilbert space operator $T$ as an extension of some operator transforms based on the Duggal transform $T^D$
by $M_{\lambda}(T) := \lambda T + (1-\lambda)T^D$, and present some of its essentially properties.
Among other things, we obtain estimates for the operator norm and numerical radius of
the $\lambda$-mean transform $M_{\lambda}(T)$ in terms of the original operator $T$.
\end{abstract}
\maketitle
\section{Introduction}
Let $\mathbb{B}(\mathcal{H})$ denote the $C^{\ast}$-algebra of all bounded
linear operators on a complex Hilbert space $\big(\mathcal{H}, \langle \cdot, \cdot\rangle \big)$.
Let the symbol $I$ stand for the identity operator on $\mathcal{H}$.
For every $T\in\mathbb{B}(\mathcal{H})$ its range is denoted by $\mathcal{R}(T)$,
its null space by $\mathcal{N}(T)$, and its adjoint by $T^*$.
We also let $\sigma(T)$ denote the spectrum of $T$, and $r(T)$ denote its spectral radius.
An operator $T\in\mathbb{B}(\mathcal{H})$ is partial isometry when $TT^*T = T$ (or, equivalently,
$T^*T$ is an orthogonal projection). In particular, $T$ is an isometry if $T^*T = I$,
and $T$ is unitary if it is a surjective isometry.
For an operator $T\in\mathbb{B}(\mathcal{H})$, there exists a unique polar decomposition
$T = U|T|$ (called the canonical polar decomposition), where $|T| = (T^*T )^{1/2}$
and $U$ is the appropriate partial isometry satisfying $\mathcal{N}(U) = \mathcal{N}(T)$.

The numerical range and the numerical radius of $T\in\mathbb{B}(\mathcal{H})$ are defined by
$W(T) = \big\{\langle Tx, x\rangle:\,\,\, x\in\mathcal{H},\,\|x\| =1\big\}$,
and
$\omega(T) = \sup\{|\xi|: \,\,\, \xi\in W(T)\}$,
respectively.
It is well known that $w(\cdot)$ defines a norm on
$\mathbb{B}(\mathcal{H})$ such that for all $T\in\mathbb{B}(\mathcal{H})$,
\begin{align}\label{1.1}
\max\big\{r(T), \frac{1}{2}\|T\|\big\}\leq \omega(T)\leq \|T\|.
\end{align}
It is known that $w(\cdot)$ is weakly unitarily invariant,
in the sense that for every $T\in\mathbb{B}(\mathcal{H})$ and unitary $V\in\mathbb{B}(\mathcal{H})$, we have
$w(V^*TV) = \omega(T)$. It is also known that for $T\in\mathbb{B}(\mathcal{H})$,
\begin{align}\label{1.2}
\omega(T) = \displaystyle{\sup_{\theta \in \mathbb{R}}}\big\|\mathfrak{Re}(e^{i\theta}T)\big\|.
\end{align}
For proofs and more facts about the numerical radius, we refer the reader to \cite{A.K.1, Halmos, Y.2}, and the references therein.

Let us introduce some transforms of Hilbert space operators.
Let $T = U|T|$ be the canonical polar decomposition of $T\in\mathbb{B}(\mathcal{H})$.
The Aluthge transform $\widetilde{T}$ of $T$ is defined by $\widetilde{T} := |T|^{1/2}U|T|^{1/2}$.
This transformation arose in the study of hyponormal operators \cite{Alu} and has since
been studied in many different contexts (see, e.g., \cite{Be.C.K.L, F.J.K.P, Y.1}).
The Duggal transform $T^D$ of $T$ is defined by $T^D := |T|U$, which is first referred to in \cite{F.J.K.P}.
The mean transform $\widehat{T}$ of $T$ is defined by $\widehat{T} := \frac{1}{2}(T + T^D)$.
This transform was first introduced in \cite{L.L.Y}
and has received much attention in recent years (see \cite{C.C.M, Ja.So, SJ.K.P, Y.3}).
A kind of operator transform
is the generalized mean transform
$\widehat{T}(t)$ of $T$, introduced recently in \cite{Be.C.K.L}, by
\begin{align*}
\widehat{T}(t): = \frac{1}{2}(|T|^{t}U|T|^{1-t} + |T|^{1-t}U|T|^{t}),
\end{align*}
for $t\in [0, 1/2]$. Clearly, $\widehat{T}(0)= \widehat{T}$ and $\widehat{T}(1/2)= \widetilde{T}$.

For more information about the transforms and their
properties, interested readers are referred to \cite{Be.C.K.L, C.C.M, F.J.K.P, SJ.K.P, L.L.Y}.

Now, we introduce a new transform of the given operator $T\in\mathbb{B}(\mathcal{H})$ based on the Duggal transform $T^D$.
\begin{definition}\label{D.0002}
Let $T = U|T|$ be the canonical polar decomposition of $T\in\mathbb{B}(\mathcal{H})$.
For $\lambda \in [0, 1]$, the $\lambda$-mean transform $M_{\lambda}(T)$ of $T$ is defined by
\begin{align*}
M_{\lambda}(T) := \lambda T + (1-\lambda)T^D,
\end{align*}
where $T^D = |T|U$ denotes the Duggal transform of $T$.
In particular, $M_{0}(T) = T^D$ and $M_{1/2}(T) = \widehat{T}$ is the mean transform of $T$.
\end{definition}
For $t\in (0, 1/2)$, it is so hard to find the generalized mean transform
$\widehat{T}(t)$ of the given operator $T\in\mathbb{B}(\mathcal{H})$ because it involves $|T|^t$,
and it is quite difficult to find $|T|^t$ in general. By contrast,
for $\lambda \in [0, 1]$, the $\lambda$-mean transform $M_{\lambda}(T)$ of $T$
involves $T^D$, so it is easy to get $M_{\lambda}(T)$ if we know the canonical polar decomposition of $T$.
Hence the $\lambda$-mean transform $M_{\lambda}(T)$ is more convenient than
the generalized mean transform $\widehat{T}(t)$ in the practical use.

The main aim of the present work is to investigate the $\lambda$-mean transform of Hilbert space operators.

In Section 2, by using some ideas of \cite{C.C.M, L.L.Y}, we first give basic properties of the $\lambda$-mean transform $M_{\lambda}(T)$
and then we provide various connections between $T$ and $M_{\lambda}(T)$ in some particular cases.

In Section 3 we derive some norm inequalities and equalities for
the $\lambda$-mean transform $M_{\lambda}(T)$ of $T$.
In particular, we obtain new estimates for the operator norm
of the mean transform of Hilbert space operators.

In the last section many numerical radius inequalities for
the $\lambda$-mean transform in terms of the original operator are given.
Our results generalize recent the norm and numerical radius inequalities of
the mean transform of operators due to Chabbabi et al. \cite{C.C.M}.
Finally, we show that the numerical radius of the sequence of the iterated $\lambda$-mean
transform of a rank one operator converges to its spectral radius.
\section{Basic properties of the $\lambda$-mean transform $M_{\lambda}(T)$}
In this section, we give some properties of the $\lambda$-mean transform of Hilbert space operators.

Recall that an operator $T\in\mathbb{B}(\mathcal{H})$ is said to be normal if $TT^*=T^*T$,
quasinormal if $T$ commutes with $T^*T$, and hyponormal if $TT^*\leq T^*T$.
It is easy to see that normal $\Longrightarrow$ quasinormal $\Longrightarrow$ hyponormal
and the inverse implications do not hold (see, for example, \cite{Halmos}).
\begin{proposition}\label{P.1002}
Let $T\in\mathbb{B}(\mathcal{H})$.
Then the following conditions are equivalent:
\begin{itemize}
\item[(i)] $T$ is quasinormal.
\item[(ii)] $M_{\lambda}(T) = T$ for all $\lambda \in [0, 1]$.
\item[(iii)] $M_{\lambda}(T) = T$ for some $\lambda \in [0, 1)$.
\end{itemize}
\end{proposition}
\begin{proof}
Let $T = U|T|$ be the canonical polar decomposition of $T$.

(i)$\Rightarrow$(ii)
Let $T$ be quasinormal.
The following fact may be known in the literature but we present its proof for readers' convenience.

Since $T$ commutes with $T^*T$, we have $(T^*T - TT^*)T = 0$ and then $(|T|^2 - |T^*|^2)U|T| = 0$.
Thus $(|T|^2 - U|T|^2U^*)U|T| = 0$. Therefore
\begin{align*}
\mathcal{R}(U^*) = \mathcal{R}(|T|) \subseteq \mathcal{N}\big((|T|^2 - U|T|^2U^*)U\big),
\end{align*}
and so $(|T|^2 - U|T|^2U^*)UU^* = 0$. Since $U^*UU^* = U^*$, we obtain $|T|^2UU^* - U|T|^2U^*=0$
and hence $(|T|^2UU^* - U|T|^2U^*)U = 0$. From this it follows that $|T|^2U - U|T|^2 = 0$,
or equivalently, $|T|$ commutes with $U$.
Thus $T^D = T$. This implies $M_{\lambda}(T) = \lambda T + (1-\lambda)T^D = T$ for all $\lambda \in [0, 1]$.

(ii)$\Rightarrow$(iii) The implication is trivial.

(iii)$\Rightarrow$(i) Suppose that $M_{\lambda}(T) = T$ for some $\lambda \in [0, 1)$.
Then $T^D = T$, or equivalently, $|T|U = U|T|$. Therefore,
\begin{align*}
T(T^*T) = U|T|^3 = |T|U|T|^2 = |T|^2U|T| = (T^*T)T,
\end{align*}
and so $T$ is quasinormal.
\end{proof}
For $\lambda \in [0, 1]$, the iterated $\lambda$-mean transforms of an operator $T\in\mathbb{B}(\mathcal{H})$
are the operators $M^{(n)}_{\lambda}(T)\,(n\geq 0)$, defined by setting $M^{(0)}_{\lambda}(T) = T$ and letting
$M^{(n+1)}_{\lambda}(T)$ be the $\lambda$-mean transform of $M^{(n)}_{\lambda}(T)$.
\begin{remark}\label{R.2002}
Let $T\in\mathbb{B}(\mathcal{H})$ and let $\lambda \in [0, 1)$.
If $T$ is quasinormal, the by Proposition \ref{P.1002} we have $M_{\lambda}(T) = T$.
Therefore $M^{(n)}_{\lambda}(T) = T$ for every $n\in \mathbb{Z}_{+}$.
\end{remark}
\begin{proposition}\label{P.3002}
Let $T\in\mathbb{B}(\mathcal{H})$.
Then the following conditions are equivalent:
\begin{itemize}
\item[(i)] $T$ is quasinormal.
\item[(ii)] $M_{\lambda}(T) = T^D$ for all $\lambda \in [0, 1]$.
\item[(iii)] $M_{\lambda}(T) = T^D$ for some $\lambda \in (0, 1]$.
\end{itemize}
\end{proposition}
\begin{proof}
The proof is similar to that of proposition \ref{P.1002} and so we omit it.
\end{proof}
\begin{remark}\label{R.3.1.002}
Let $\lambda \in [0, 1]$. By \cite[p.79]{L.L.Y},
the Duggal transform map $T\longrightarrow T^D$ is $(\|\cdot\|, SOT)$-continuous
on $\mathbb{B}(\mathcal{H})$. Therefore, since $M_{\lambda}(T) = \lambda T + (1-\lambda)T^D$,
the $\lambda$-mean transform map $T\longrightarrow M_{\lambda}(T)$
is also $(\|\cdot\|, SOT)$-continuous on $\mathbb{B}(\mathcal{H})$.
\end{remark}
Next, we recall the class of unilateral weighted shifts. For $\alpha \equiv \{\alpha_n\}^{\infty}_{n=0}$
a bounded sequence of positive numbers (called weights), let
\begin{align*}
\mathcal{W}_\alpha = \,\mbox{shift}(\alpha_0, \alpha_1, \cdots): \ell^2(\mathbb{Z}_{+})\longrightarrow \ell^2(\mathbb{Z}_{+})
\end{align*}
be the associated unilateral weighted shift, defined by $\mathcal{W}_\alpha e_{n}: = \alpha_{n}e_{n+1}$, where
$\{e_n\}^{\infty}_{n=0}$ is the canonical orthonormal basis in $\ell^2(\mathbb{Z}_{+})$.
For a weighted shift $\mathcal{W}_\alpha$, it is easy to see that $\mathcal{W}_\alpha$
is hyponormal if and only if $\{\alpha_n\}^{\infty}_{n=0}$ is an increasing sequence.
The spectral radius and spectrum of $\mathcal{W}_\alpha$ are well known:
\begin{align}\label{P.4002.1}
r(\mathcal{W}_\alpha) = \displaystyle{\lim_{n\rightarrow +\infty}}\Big(\displaystyle{\sup_{k\in \mathbb{N}}}
\big(\alpha_{k}\ldots\alpha_{k+n-1}\big)\Big)^{1/n}
\end{align}
and
\begin{align}\label{P.4002.2}
\sigma(\mathcal{W}_\alpha) = \big\{\gamma \in \mathbb{C}: \, |\gamma| \leq r(\mathcal{W}_\alpha)\big\}.
\end{align}
The Aluthge transform $\widetilde{\mathcal{W}_\alpha}$ of $\mathcal{W}_\alpha$ is also a unilateral weighted shift:
\begin{align}\label{P.4002.3}
\widetilde{\mathcal{W}_\alpha} = \,\mbox{shift}(\sqrt{\alpha_0\alpha_1}, \sqrt{\alpha_1\alpha_2}, \ldots).
\end{align}
For proofs and more facts about the weighted shift operators, we refer the reader to \cite{Halmos, Shie}.
\begin{proposition}\label{P.4002}
Let $\alpha \equiv \{\alpha_n\}^{\infty}_{n=0}$ be a sequence of positive numbers, and let
$\mathcal{W}_\alpha$ be the associate weighted shift. For every $\lambda \in [0, 1]$,
the following statements are equivalent:
\begin{itemize}
\item[(i)] $M_{\lambda}(\mathcal{W}_\alpha)$ is hyponormal.
\item[(ii)] $\lambda\big(\alpha_n - \alpha_{n+1}\big) \leq (1-\lambda)\big(\alpha_{n+2} - \alpha_{n+1}\big)$
for all $n\in \mathbb{Z}_{+}$.
\end{itemize}
\end{proposition}
\begin{proof}
We first note that the polar decomposition of $\mathcal{W}_\alpha$ is $U|\mathcal{W}_\alpha|$,
where $U$ is the unilateral shift and $|\mathcal{W}_\alpha|=\mbox{diag}(\alpha_0, \alpha_1, \ldots)$.
Hence the $\lambda$-mean transform of $\mathcal{W}_\alpha$ is the following weighted shift
operator:
\begin{align}\label{I.P.4002}
M_{\lambda}(\mathcal{W}_\alpha) e_{n} = \big(\lambda \alpha_n + (1-\lambda)\alpha_{n+1}\big)e_{n+1} \qquad (n\in \mathbb{Z}_{+}).
\end{align}
Therefore, $M_{\lambda}(\mathcal{W}_\alpha)$ is hyponormal if and only if
\begin{align*}
\lambda \alpha_n + (1-\lambda)\alpha_{n+1} \leq \lambda \alpha_{n+1} + (1-\lambda)\alpha_{n+2}
\end{align*}
for all $n\in \mathbb{Z}_{+}$, or equivalently,
\begin{align*}
\lambda\big(\alpha_n - \alpha_{n+1}\big) \leq (1-\lambda)\big(\alpha_{n+2} - \alpha_{n+1}\big)
\end{align*}
for all $n\in \mathbb{Z}_{+}$.
\end{proof}
\begin{remark}\label{R.5002}
Let $\alpha \equiv \{\alpha_n\}^{\infty}_{n=0}$ be a sequence of positive numbers, and let
$\mathcal{W}_\alpha$ be the associate weighted shift.
For every $\lambda \in [0, 1]$, if $\mathcal{W}_\alpha$ is hyponormal, then $\alpha_n \leq \alpha_{n+1}$ for all $n\in \mathbb{Z}_{+}$.
Therefore
$\lambda\big(\alpha_n - \alpha_{n+1}\big) \leq (1-\lambda)\big(\alpha_{n+2} - \alpha_{n+1}\big)$
for all $n\in \mathbb{Z}_{+}$. So, by proposition \ref{P.4002}, $M_{\lambda}(\mathcal{W}_\alpha)$ is hyponormal.
The converse does not hold in general. For example, let $\mathcal{W}_\beta$
be the unilateral weighted shift defined by $\beta \equiv \{\beta_n\}^{\infty}_{n=0}$, where
$\beta_n : =\left\{\begin{array}{ll}
1 &n\neq1\\
1/2 &n=1\end{array}\right.$ and let $\lambda = 1/3$.
Hence, by Proposition \ref{P.4002}, we know that $M_{\lambda}(\mathcal{W}_\beta)$ is hyponormal.
However, $\mathcal{W}_\beta$ is not hyponormal since $\beta_n \leq \beta_{n+1}$ does not hold.
\end{remark}
\begin{proposition}\label{P.4n002}
Let $\alpha \equiv \{\alpha_n\}^{\infty}_{n=0}$ be a sequence of positive numbers, and let
$\mathcal{W}_\alpha$ be the associate weighted shift.
For every $\lambda \in [0, 1)$, the following statements hold.
\begin{itemize}
\item[(i)] If $\alpha$ is monotone decreasing, then
\begin{align*}
\displaystyle{\lim_{m\rightarrow +\infty}} M^{(m)}_{\lambda}(\mathcal{W}_\alpha)
= (\inf_{n\geq0}\alpha_n)\,\mbox{shift}(1, 1, \cdots).
\end{align*}
\item[(ii)] If $\alpha$ is monotone increasing, then
\begin{align*}
\displaystyle{\lim_{m\rightarrow +\infty}} M^{(m)}_{\lambda}(\mathcal{W}_\alpha)
= (\sup_{n\geq0}\alpha_n)\,\mbox{shift}(1, 1, \cdots).
\end{align*}
\end{itemize}
\end{proposition}
\begin{proof}
(i) We first exhibit the $m$-th iterated $\lambda$-mean transform $M^{(m)}_{\lambda}(\mathcal{W}_\alpha)$.
By (\ref{I.P.4002}), we obtain that
\begin{align*}
M^{(2)}_{\lambda}(\mathcal{W}_\alpha) e_{n}
= \Big(\lambda \big(\lambda \alpha_n + (1-\lambda)\alpha_{n+1}\big)
+ (1-\lambda)\big(\lambda \alpha_{n+1} + (1-\lambda)\alpha_{n+2}\big)\Big)e_{n+1}
\end{align*}
and so
\begin{align*}
M^{(2)}_{\lambda}(\mathcal{W}_\alpha) e_{n}
= \Big(\lambda^2 \alpha_n + 2\lambda(1-\lambda)\alpha_{n+1} + (1-\lambda)^2\alpha_{n+2}\Big)e_{n+1}. \qquad (n\in \mathbb{Z}_{+})
\end{align*}
Induction on $m$ shows therefore that
\begin{align}\label{P.4n002.I1}
M^{(m)}_{\lambda}(\mathcal{W}_\alpha) e_{n}
= \Big(\sum_{i=0}^{m}\Big(\begin{array}{c}
m\\i\\ \end{array}\Big)
\lambda^{m-i}(1-\lambda)^i\alpha_{n+i}\Big)e_{n+1}. \qquad (n\in \mathbb{Z}_{+})
\end{align}
On the other hand, since $\alpha$ is monotone decreasing, we have
$\displaystyle{\inf_{n\geq0}}\alpha_n < \alpha_{i}$ for all $i$.
Therefore,
\begin{align*}
0 &\leq\sum_{i=0}^{m}\Big(\begin{array}{c}
m\\i\\ \end{array}\Big)
\lambda^{m-i}(1-\lambda)^i(\alpha_{i} - \inf_{n\geq0}\alpha_n)
\\&=  \sum_{i=0}^{m}\Big(\begin{array}{c}
m\\i\\ \end{array}\Big)
\lambda^{m-i}(1-\lambda)^i\alpha_{i}
- \inf_{n\geq0}\alpha_n\Big(\sum_{i=0}^{m}\Big(\begin{array}{c}
m\\i\\ \end{array}\Big)
\lambda^{m-i}(1-\lambda)^i\Big)
\\& = \sum_{i=0}^{m}\Big(\begin{array}{c}
m\\i\\ \end{array}\Big)
\lambda^{m-i}(1-\lambda)^i\alpha_{i}
- \inf_{n\geq0}\alpha_n,
\end{align*}
so (\ref{P.4n002.I1}) implies
\begingroup\makeatletter\def\f@size{10}\check@mathfonts
\begin{align}\label{P.4n002.I2}
\Big\|M^{(m)}_{\lambda}(\mathcal{W}_\alpha)
-(\inf_{n\geq0}\alpha_n)\,\mbox{shift}(1, 1, \cdots)\Big\| = \sum_{i=0}^{m}\Big(\begin{array}{c}
m\\i\\ \end{array}\Big)
\lambda^{m-i}(1-\lambda)^i(\alpha_{i} - \inf_{n\geq0}\alpha_n).
\end{align}
\endgroup
Also, for every $1\leq k \leq m$, we have
\begin{align*}
0 &\leq \sum_{i=k}^{m}\Big(\begin{array}{c}
m\\i\\ \end{array}\Big)
\lambda^{m-i}(1-\lambda)^i(\alpha_{i} - \inf_{n\geq0}\alpha_n)
\\& \leq (\alpha_{k} - \inf_{n\geq0}\alpha_n)\sum_{i=k}^{m}\Big(\begin{array}{c}
m\\i\\ \end{array}\Big)
\lambda^{m-i}(1-\lambda)^i
\\& \leq (\alpha_{k} - \inf_{n\geq0}\alpha_n)\sum_{i=0}^{m}\Big(\begin{array}{c}
m\\i\\ \end{array}\Big)
\lambda^{m-i}(1-\lambda)^i,
\end{align*}
whence
\begin{align}\label{P.4n002.I3}
0 \leq \sum_{i=k}^{m}\Big(\begin{array}{c}
m\\i\\ \end{array}\Big)
\lambda^{m-i}(1-\lambda)^i(\alpha_{i} - \inf_{n\geq0}\alpha_n)
\leq (\alpha_{k} - \inf_{n\geq0}\alpha_n).
\end{align}
Further,
\begin{align*}
0 &\leq \sum_{i=0}^{k-1}\Big(\begin{array}{c}
m\\i\\ \end{array}\Big)
\lambda^{m-i}(1-\lambda)^i(\alpha_{i} - \inf_{n\geq0}\alpha_n)
\\& \leq (\alpha_{0} - \inf_{n\geq0}\alpha_n)\sum_{i=0}^{k-1}\Big(\begin{array}{c}
m\\i\\ \end{array}\Big)
\lambda^{m-i}(1-\lambda)^i
\\& \leq (\alpha_{0} - \inf_{n\geq0}\alpha_n)\max\big\{\lambda^m, (1-\lambda)^m\big\}
\sum_{i=0}^{k-1}\Big(\begin{array}{c}
m\\i\\ \end{array}\Big)
\\& \leq (\alpha_{0} - \inf_{n\geq0}\alpha_n)\max\big\{\lambda^m, (1-\lambda)^m\big\}
k\Big(\begin{array}{c}
m\\k\\ \end{array}\Big)
\\& \leq (\alpha_{0} - \inf_{n\geq0}\alpha_n)\max\big\{\lambda^m, (1-\lambda)^m\big\}
k\frac{m!}{(m-k)!},
\end{align*}
and so
\begingroup\makeatletter\def\f@size{10}\check@mathfonts
\begin{align}\label{P.4n002.I4}
0 \leq \sum_{i=0}^{k-1}\Big(\begin{array}{c}
m\\i\\ \end{array}\Big)
\lambda^{m-i}(1-\lambda)^i(\alpha_{i} - \inf_{n\geq0}\alpha_n)
\leq (\alpha_{0} - \inf_{n\geq0}\alpha_n)\max\big\{\lambda^m, (1-\lambda)^m\big\}
k\frac{m!}{(m-k)!}.
\end{align}
\endgroup
Now, by (\ref{P.4n002.I2}), (\ref{P.4n002.I3}) and (\ref{P.4n002.I4}), we conclude that
\begin{align*}
\displaystyle{\lim_{m\rightarrow +\infty}}\Big\|M^{(m)}_{\lambda}(\mathcal{W}_\alpha)
-(\inf_{n\geq0}\alpha_n)\,\mbox{shift}(1, 1, \cdots)\Big\| = 0,
\end{align*}
and hence $\displaystyle{\lim_{m\rightarrow +\infty}} M^{(m)}_{\lambda}(\mathcal{W}_\alpha)
= (\inf_{n\geq0}\alpha_n)\,\mbox{shift}(1, 1, \cdots)$.

The proof of (ii) is so similar to that of (i) that we omit it.
\end{proof}
Let $T\in\mathbb{B}(\mathcal{H})$ and let $\lambda \in (0, 1)$.
Since $\sigma(T^D) = \sigma(T)$ (see, for example, \cite{Y.1}) and $M_{\lambda}(T) = \lambda T + (1-\lambda)T^D$,
one might be tempted to guess that $\sigma(T) \subseteq \sigma\big(M_{\lambda}(T)\big)$.
But, below, we will show that there exists an operator $T\in\mathbb{B}(\mathcal{H})$,
which neither $\sigma(T) \subseteq \sigma\big(M_{\lambda}(T)\big)$
nor $\sigma\big(M_{\lambda}(T)\big) \subseteq \sigma(T)$.
\begin{example}\label{E.6002}
Let $T = \begin{bmatrix}
0 & 0 & 1 & 0\\
0 & 0 & 0 & 2\\
0 & 0 & 0 & 0\\
0 & 0 & 0 & 0
\end{bmatrix}$. It is easy to see that
$\sigma(T) = \{0\}$.
Also, simple computations show that
$T^D = \begin{bmatrix}
0 & 0 & 0 & 0\\
0 & 0 & 0 & 0\\
1 & 0 & 0 & 0\\
0 & 2 & 0 & 0
\end{bmatrix}$.
Therefore for $\lambda \in(0, 1)$, we have
\begin{align*}
M_{\lambda}(T) = \begin{bmatrix}
0 & 0 & \lambda & 0\\
0 & 0 & 0 & 2\lambda\\
1-\lambda & 0 & 0 & 0\\
0 & 2(1-\lambda) & 0 & 0
\end{bmatrix}.
\end{align*}
Finally, a direct calculation shows that
$\sigma\big(M_{\lambda}(T)\big) = \big\{\pm \sqrt{\lambda-\lambda^2}, \pm 2\sqrt{\lambda-\lambda^2}\big\}$.
\end{example}
For the unilateral weighted shifts, we have the following result.
\begin{theorem}\label{T.7002}
Let $\alpha \equiv \{\alpha_n\}^{\infty}_{n=0}$ be a sequence of positive numbers, and let
$\mathcal{W}_\alpha$ be the associate weighted shift. For every $\lambda \in (0, 1)$,
\begingroup\makeatletter\def\f@size{10}\check@mathfonts
\begin{align*}
2\sqrt{\lambda - \lambda^2}\,\sigma(\mathcal{W}_\alpha)
:= \Big\{2\sqrt{\lambda - \lambda^2}\,\gamma: \, \gamma\in \sigma(\mathcal{W}_\alpha)\Big\}
\subseteq \,\sigma\big(M_{\lambda}(\mathcal{W}_\alpha)\big).
\end{align*}
\endgroup
\end{theorem}
\begin{proof}
For $A\in\mathbb{B}(\mathcal{H})$, it is well known that the spectral radius of the Aluthge transform $A$
equals that of $A$ (see, for example, \cite{Y.1}).
Therefore, by (\ref{P.4002.1}) and (\ref{P.4002.3}), we have
\begingroup\makeatletter\def\f@size{10}\check@mathfonts
\begin{align*}
2\sqrt{\lambda - \lambda^2}\,r(\mathcal{W}_\alpha) &= 2\sqrt{\lambda - \lambda^2}\,r(\widetilde{\mathcal{W}_\alpha})
\\ & = 2\sqrt{\lambda - \lambda^2}\displaystyle{\lim_{n\rightarrow +\infty}}\Big(\displaystyle{\sup_{k\in \mathbb{N}}}
\big(\sqrt{\alpha_{k}\alpha_{k+1}}\ldots \sqrt{\alpha_{k+n-1}\alpha_{k+n}}\big)\Big)^{1/n}
\\ & = 2\displaystyle{\lim_{n\rightarrow +\infty}}\Big(\displaystyle{\sup_{k\in \mathbb{N}}}
\big(\sqrt{\lambda\alpha_{k}(1-\lambda)\alpha_{k+1}}\ldots
\sqrt{\lambda\alpha_{k+n-1}(1-\lambda)\alpha_{k+n}}\big)\Big)^{1/n}
\\& \leq \displaystyle{\lim_{n\rightarrow +\infty}}\Big(\displaystyle{\sup_{k\in \mathbb{N}}}
\Big(\big(\lambda\alpha_{k}+(1-\lambda)\alpha_{k+1}\big)\ldots
\big(\lambda\alpha_{k+n-1}+(1-\lambda)\alpha_{k+n}\big)\Big)\Big)^{1/n}
\\& \qquad \qquad \qquad \qquad \big(\mbox{by the arithmetic-geometric mean inequality}\big)
\\& = r\big(M_{\lambda}(\mathcal{W}_\alpha)\big),
\end{align*}
\endgroup
whence
\begin{align}\label{T.7002.1}
2\sqrt{\lambda - \lambda^2}\,r(\mathcal{W}_\alpha) \leq r\big(M_{\lambda}(\mathcal{W}_\alpha)\big).
\end{align}
Now from (\ref{P.4002.2}) and (\ref{T.7002.1}) it follows that
\begin{align*}
\Big\{2\sqrt{\lambda - \lambda^2}\,\gamma: \, \gamma\in \sigma(\mathcal{W}_\alpha)\Big\}
\subseteq \,\sigma\big(M_{\lambda}(\mathcal{W}_\alpha)\big).
\end{align*}
\end{proof}
\begin{remark}\label{R.8002}
Let $T = U|T|$ be the canonical polar decomposition of $T\in\mathbb{B}(\mathcal{H})$ and let $\lambda \in (0, 1)$.
If $T$ is invertible, then $U$ is unitary and so $\lambda|T|$ and $(1-\lambda)U^*|T|U$ are positive and invertible.
Therefore,
\begin{align*}
M_{\lambda}(T) = \lambda T + (1-\lambda)T^D = U\big(\lambda |T| + (1-\lambda)U^*|T|U)
\end{align*}
is also invertible. The converse is however not true, in general.
For example (see \cite[Example 2.4]{C.C.M}), let $\mathcal{W}_\beta$ be the weighted bilateral shift defined by
$\mathcal{W}_\beta(e_n) = \beta_{n}e_{n+1}$ for all $n\in \mathbb{Z}$, where
$\beta_n : =\left\{\begin{array}{ll}
1 &\mbox{if n even}\\
\frac{1}{n^2} &\mbox{if n odd}.\end{array}\right.$

Since $\mathcal{W}_\beta(e_{2n+1}) = \beta_{2n+1}e_{2n+2} = \frac{1}{(2n+1)^2}e_{2n+2}$, we get
$\displaystyle{\lim_{n\rightarrow +\infty}}\mathcal{W}_\beta(e_{2n+1}) = 0$ and hence $T$ is not invertible.
On the other hand, we have
\begin{align*}
\lambda\beta_n  + (1-\lambda)\beta_{n+1} =\left\{\begin{array}{ll}
\lambda + \frac{1-\lambda}{(n+1)^2} &\mbox{if n even}\\
1-\lambda + \frac{\lambda}{n^2} &\mbox{if n odd}.\end{array}\right.
\end{align*}
Thus $\max\{\lambda, 1-\lambda\} \leq \lambda\beta_n  + (1-\lambda)\beta_{n+1} \leq 1$
for all $n\in \mathbb{Z}$, and so $M_{\lambda}(\mathcal{W}_\beta)$ is invertible.
\end{remark}
In the last part of this section we consider the symmetric of the $\lambda$-mean
transform of some complex symmetric truncated weighted shifts.
Recall that a conjugation on a complex Hilbert space $\mathcal{H}$
is a conjugate-linear operator $C : \mathcal{H} \longrightarrow \mathcal{H}$,
which is both involutive ($C^2 = I$) and isometric. An operator $T\in\mathbb{B}(\mathcal{H})$
is said to be complex symmetric if there exists a conjugation $C$ on
$\mathcal{H}$ such that $CTC = T^*$. Many usual operators such as normal
operators, compressed Toeplitz operators, Hankel matrices, and the Volterra integration operators are included in the
class of complex symmetric operators. We refer the reader to \cite{G.Pu, G.Wo} for further details.
\begin{lemma}\cite[Proposition 3.2]{Z.L}\label{L.9002}
Let ${\{e_n\}}^m_{n=1}$ be an orthonormal basis of $\mathbb{C}^m$. If
$T= \sum_{n=1}^{m-1}\alpha_ne_n\otimes e_{n+1}$ and $\alpha_n\neq0$ for all $1\leq n \leq m-1$.
Then $T$ is complex symmetric if and only if $|\alpha_n| = |\alpha_{m-n}|$ for all $1\leq n \leq m-1$.
\end{lemma}
\begin{theorem}\label{T.10.002}
Let ${\{e_n\}}^m_{n=1}$ be an orthonormal basis of $\mathbb{C}^m$ and $\lambda\in(0, 1)$.
If $T= \sum_{n=1}^{m-1}\alpha_ne_n\otimes e_{n+1}$ where $m\geq3$ and $\alpha_n$ are complex numbers
such that $\alpha_1 \neq 0$ and $\lambda|\alpha_{n+1}| + (1-\lambda)|\alpha_n| \neq 0$
for all $1\leq n \leq m-2$, then the following statements are equivalent:
\begin{itemize}
\item[(i)] $M_{\lambda}(T)$ is complex symmetric.
\item[(ii)]
\begingroup\makeatletter\def\f@size{10}\check@mathfonts
$\lambda\big(|\alpha_1| - |\alpha_{m-1}|\big) = (1-\lambda)|\alpha_{m-2}|$
and
$\lambda\big(|\alpha_n| - |\alpha_{m-n}|\big) = (1-\lambda)\big(|\alpha_{m-n-1}| - |\alpha_{n-1}|\big)$
for all $2\leq n \leq m-2$.
\endgroup
\end{itemize}
\end{theorem}
\begin{proof}
For $1\leq n \leq m-1$, let us put $\alpha _n = e^{i\theta_n}|\alpha_n|$ for some real number $\theta_n$.
Suppose that
\begin{align}\label{T.10.002.1}
T = \begin{bmatrix}
0 & \alpha_1 & 0 & \cdot & \cdots & 0\\
0 & 0 & \alpha_2 & 0 & \cdots & 0\\
\vdots & \vdots & \ddots & \ddots & \cdots & \vdots\\
\cdot & \cdot & \cdots & 0 & \cdot & 0\\
\cdot & \cdot & \cdots & \cdot & 0 & \alpha_{m-1}\\
0 & 0 & \cdots & \cdot & \cdot & 0
\end{bmatrix}
\end{align}
and
\begin{align*}
U = \begin{bmatrix}
0 & e^{i\theta_1} & 0 & \cdot & \cdots & 0\\
0 & 0 & e^{i\theta_2} & 0 & \cdots & 0\\
\vdots & \vdots & \ddots & \ddots & \cdots & \vdots\\
\cdot & \cdot & \cdots & 0 & \cdot & 0\\
\cdot & \cdot & \cdots & \cdot & 0 & e^{i\theta_{m-1}}\\
0 & 0 & \cdots & \cdot & \cdot & 0
\end{bmatrix}.
\end{align*}
It is easy to see that $|T|=\mbox{diag}(0, |\alpha_0|, |\alpha_1|, \ldots, |\alpha_{m-1}|)$
and $T = U|T|$ is the canonical polar decomposition of $T$.
Therefore,
\begin{align}\label{T.10.002.2}
T^D = |T|U =  = \begin{bmatrix}
0 & 0 & 0 & \cdot & \cdots & 0\\
0 & 0 & e^{i\theta_2}|\alpha_1| & 0 & \cdots & 0\\
\vdots & \vdots & \ddots & \ddots & \cdots & \vdots\\
\cdot & \cdot & \cdots & 0 & \cdot & 0\\
\cdot & \cdot & \cdots & \cdot & 0 & e^{i\theta_{m-1}}|\alpha_{m-2}|\\
0 & 0 & \cdots & \cdot & \cdot & 0
\end{bmatrix}.
\end{align}
Since $M_{\lambda}(T) = \lambda T + (1-\lambda)T^D$, by (\ref{T.10.002.1}) and (\ref{T.10.002.2}), we obtain
\begingroup\makeatletter\def\f@size{8}\check@mathfonts
\begin{align*}
M_{\lambda}(T) = \begin{bmatrix}
0 & \lambda \alpha_1 & 0 & \cdot & \cdots & 0\\
0 & 0 & e^{i\theta_2}\big(\lambda|\alpha_2| + (1-\lambda)|\alpha_1|\big) & 0 & \cdots & 0\\
\vdots & \vdots & \ddots & \ddots & \cdots & \vdots\\
\cdot & \cdot & \cdots & 0 & \cdot & 0\\
\cdot & \cdot & \cdots & \cdot & 0 & e^{i\theta_{m-1}}\big(\lambda|\alpha_{m-1}| + (1-\lambda)|\alpha_{m-2}|\big)\\
0 & 0 & \cdots & \cdot & \cdot & 0
\end{bmatrix}.
\end{align*}
\endgroup
Now Lemma \ref{L.9002} implies that $M_{\lambda}(T)$ is complex symmetric if and only if
$\lambda\big(|\alpha_1| - |\alpha_{m-1}|\big) = (1-\lambda)|\alpha_{m-2}|$ and
$\lambda\big(|\alpha_n| - |\alpha_{m-n}|\big) = (1-\lambda)\big(|\alpha_{m-n-1}| - |\alpha_{n-1}|\big)$ for all $2\leq n \leq m-2$.
\end{proof}
\begin{remark}
In general, the $\lambda$-mean transform $M_{\lambda}(T)$ of a complex symmetric operator $T$
may not be complex symmetric.
For example, let $\lambda\in(0, 1)$ and let
\begin{align*}
T = \begin{bmatrix}
0 & 1 & 0 & 0 & 0\\
0 & 0 & 2 & 0 & 0\\
0 & 0 & 0 & 2 & 0\\
0 & 0 & 0 & 0 & 1\\
0 & 0 & 0 & 0 & 0
\end{bmatrix}.
\end{align*}
Simple computations show that
\begin{align*}
M_{\lambda}(T) = \begin{bmatrix}
0 & \lambda & 0 & 0 & 0\\
0 & 0 & \lambda+1 & 0 & 0\\
0 & 0 & 0 & 2 & 0\\
0 & 0 & 0 & 0 & 1-\lambda\\
0 & 0 & 0 & 0 & 0
\end{bmatrix}.
\end{align*}
So, by Lemma \ref{L.9002}, $T$ is complex symmetric and
$M_{\lambda}(T)$ is not complex symmetric.

On the other hand, if
$S = \begin{bmatrix}
0 & 1/{\lambda} & 0 & 0\\
0 & 0 & 1 & 0\\
0 & 0 & 0 & 1\\
0 & 0 & 0 & 0
\end{bmatrix}$ with $\lambda\in(0, 1)$,
then a direct calculation shows that
\begin{align*}
M_{\lambda}(S) = \begin{bmatrix}
0 & 1 & 0 & 0\\
0 & 0 & \lambda + 1/{\lambda}-1 & 0\\
0 & 0 & 0 & 1\\
0 & 0 & 0 & 0
\end{bmatrix}.
\end{align*}
Again by Lemma \ref{L.9002}, $M_{\lambda}(S)$
is complex symmetric and $S$ is not complex symmetric.
\end{remark}
\section{Bounds for the operator norm of the $\lambda$-mean transform}
In this section, we derive some norm inequalities and equalities for
the $\lambda$-mean transform of Hilbert space operators.

First, we state a norm inequality for the Duggal transform that will be needed.
\begin{proposition}\label{P.1003}
Let $T\in\mathbb{B}(\mathcal{H})$ and $\gamma \in \mathbb{C}$. Then
\begin{align*}
\big\|T^D - \gamma I\big\| \leq \|T - \gamma I\|.
\end{align*}
In particular, $\|T^D\| \leq \|T\|$.
\end{proposition}
\begin{proof}
Let $T=U|T|$ be the canonical polar decomposition of $T$.
Suppose first that $|T|$ is invertible. Since $T^D |T| = |T|T$, we get $T^D = |T|T{|T|}^{-1}$. Therefore,
\begin{align*}
\big\|T^D - \gamma I\big\| = \big\||T|(T-\gamma I){|T|}^{-1}\big\|
\leq \big\||T|\big\|\big\|T-\gamma I\big\|\big\|{|T|}^{-1}\big\|
\leq \|T - \alpha I\|,
\end{align*}
and so $\big\|T^D - \gamma I\big\| \leq \|T - \gamma I\|$.
If $|T|$ is not invertible, then for each $n \in \mathbb{N}$ we put
\begin{align*}
T_n = |T| + \frac{1}{n}I \qquad \mbox{and} \qquad X_n = UT_n.
\end{align*}
Since $T_n$ is positive invertible, by what was proved above, we have
\begin{align}\label{P.1003.1}
\big\|X^D_n - \gamma I\big\| \leq \|X_n - \gamma I\|.
\end{align}
We have $\|X^D_n - T^D\| = \|(T_n - |T|)U\| \leq \|T_n - |T|\|$ and hence
\begin{align}\label{P.1003.2}
\|X^D_n - T^D\| \leq \frac{1}{n}.
\end{align}
Further, $\|X_n - T\| = \|U(T_n - |T|)\| \leq \|T_n - |T|\|$ and so
\begin{align}\label{P.1003.3}
\|X_n - T\| \leq \frac{1}{n}.
\end{align}
Now, by (\ref{P.1003.1}), (\ref{P.1003.2}) and (\ref{P.1003.3}), we have
\begin{align*}
\big\|T^D - \gamma I\big\| &\leq \|X^D_n - T^D\| + \big\|X^D_n - \gamma I\big\|
\leq \frac{1}{n} + \big\|X_n - \gamma I\big\|
\\& \leq \frac{1}{n} + \|X_n - T\| + \big\|T - \gamma I\big\|
\leq \frac{2}{n} + \big\|T - \gamma I\big\|.
\end{align*}
Taking limits, we get $\big\|T^D - \gamma I\big\| \leq \|T - \gamma I\|$.
\end{proof}
Here we present one of the main results of this section.
\begin{theorem}\label{T.2003}
Let $T\in\mathbb{B}(\mathcal{H})$ and let $\lambda \in [0, 1]$.
Then
\begin{align}\label{T.2003.1}
2\sqrt{\lambda - \lambda^2}\,\|\widetilde{T}\|\leq \big\|M_{\lambda}(T)\big\| \leq \lambda\|T\| + (1-\lambda)\|T^D\big\|.
\end{align}
In particular,
\begin{align}\label{T.2003.2}
2\sqrt{\lambda - \lambda^2}\,r(T)\leq \big\|M_{\lambda}(T)\big\| \leq\|T\|.
\end{align}
\end{theorem}
\begin{proof}
Let $T=U|T|$ be the canonical polar decomposition of $T$.
Since $M_{\lambda}(T)=\lambda T + (1-\lambda)T^D$,
by the triangle inequality, we obtain
\begin{align}\label{T.2003.3}
\big\|M_{\lambda}(T)\big\| \leq \lambda\|T\| + (1-\lambda)\|T^D\big\|.
\end{align}
The Heinz inequality (see \cite{Heinz}) is the following:
for positive operators $A, B\in\mathbb{B}(\mathcal{H})$ and any $X\in\mathbb{B}(\mathcal{H})$,
it holds
\begin{align*}
\big\|A^{1/2}XB^{1/2}\big\|\leq \left\|\frac{AX + XB}{2}\right\|.
\end{align*}
Applying the Heinz inequality with $A= (1-\lambda)|T|$, $B = \lambda|T|$ and $X = U$, we get
\begin{align*}
\Big\|\sqrt{1-\lambda}|T|^{1/2}U\sqrt{\lambda}|T|^{1/2}\Big\|\leq \left\|\frac{(1-\lambda)|T|U + U\lambda|T|}{2}\right\|.
\end{align*}
Therefore,
\begin{align*}
\sqrt{\lambda - \lambda^2}\big\|\,|T|^{1/2}U|T|^{1/2}\big\|\leq \frac{1}{2}\big\|\lambda T + (1-\lambda)T^D\big\|,
\end{align*}
whence
\begin{align}\label{T.2003.4}
2\sqrt{\lambda - \lambda^2}\|\widetilde{T}\|\leq \big\|M_{\lambda}(T)\big\|.
\end{align}
Utilizing (\ref{T.2003.3}) and (\ref{T.2003.4}), we deduce the inequalities (\ref{T.2003.1}).

Further, since $r(T) = r(\widetilde{T})$, by (\ref{1.1}) we get
\begin{align}\label{T.2003.5}
r(T)\leq \|\widetilde{T}\|.
\end{align}
The inequalities (\ref{T.2003.2}) now follow from Proposition \ref{P.1003}
and the inequalities in (\ref{T.2003.1}) and (\ref{T.2003.5}).
\end{proof}
If $\|T\| = 0$, then by (\ref{T.2003.2}) we have $\big\|M_{\lambda}(T)\big\| = 0$.
In the following proposition we show that the converse also holds true for $\lambda \in (0, 1]$.
\begin{proposition}\label{P.2003}
Let $T\in\mathbb{B}(\mathcal{H})$ and let $\lambda \in (0, 1]$.
Then the following conditions are equivalent:
\begin{itemize}
\item[(i)] $\|T\| = 0$.
\item[(ii)] $\big\|M_{\lambda}(T)\big\| = 0$.
\end{itemize}
\end{proposition}
\begin{proof}
It is obvious that (i) implies (ii). We will prove that (ii) implies (i).
Let $\big\|M_{\lambda}(T)\big\| = 0$ and let $T=U|T|$ be the canonical polar decomposition of $T$.
Then $\big(\lambda U|T|+(1-\lambda)|T|U\big)x= 0$ for all $x\in\mathcal{H}$. We have
\begin{align*}
\big\||T|\big\| &= \frac{1}{\lambda}\displaystyle{\sup_{\|x\|=1}}\big\langle \lambda|T|x, x\big\rangle
\\& \leq \frac{1}{\lambda}\displaystyle{\sup_{\|x\|=1}}\Big(\big\langle \lambda|T|x, x\big\rangle + \big\langle(1-\lambda)U^*|T|Ux, x\big\rangle\Big)
\\&\qquad \qquad \qquad \qquad \qquad \big(\mbox{since $(1-\lambda)U^*|T|U$ is a positive operator}\big)
\\& = \frac{1}{\lambda}\displaystyle{\sup_{\|x\|=1}}\Big\langle \big(\lambda|T|+(1-\lambda)U^*|T|U\big)x, x\Big\rangle
\\& = \frac{1}{\lambda}\displaystyle{\sup_{\|x\|=1}}\Big\langle U^*\big(\lambda U|T|+(1-\lambda)|T|U\big)x, x\Big\rangle
\\&\qquad \qquad \qquad \qquad \qquad \qquad \qquad \qquad \qquad \qquad \big(\mbox{since $U^*U|T| = |T|$}\big)
\\& = \frac{1}{\lambda}\displaystyle{\sup_{\|x\|=1}}\Big\langle \big(\lambda U|T|+(1-\lambda)|T|U\big)x, Ux\Big\rangle = 0,
\end{align*}
and so $\big\||T|\big\| \leq 0$. Consequently, $\|T\|=\big\||T|\big\| = 0$.
\end{proof}
\begin{remark}\label{R.3003}
Let $T$ be a nonzero bounded operator on $\mathcal{H}$ such that $\widetilde{T}=0$.
Consider the canonical polar decomposition $T = U|T|$ of $T$. We have
\begin{align*}
T^2 = U|T|U|T| = U|T|^{1/2}\widetilde{T}|T|^{1/2} = 0,
\end{align*}
and hence
\begin{align*}
\mathcal{R}(U) = \overline{\mathcal{R}(T)} \subseteq \mathcal{N}(T) = \mathcal{N}(|T|).
\end{align*}
This implies $T^D = |T|U = 0$. Thus $\big\|M_{0}(T)\big\| = 0$, and while what we have $\|T\| > 0$.
\end{remark}
In the following theorem we obtain an improvement of the second inequality in (\ref{T.2003.1}).
\begin{theorem}\label{T.4003}
Let $T\in\mathbb{B}(\mathcal{H})$ and let $\lambda \in [0, 1]$. Then
\begingroup\makeatletter\def\f@size{8}\check@mathfonts
\begin{align}\label{T.4003.1}
\big\|M_{\lambda}(T)\big\| \leq \frac{\big\|\lambda|T| + (1 - \lambda)|T^D|\big\|
+ \big\|\lambda|T^*| + (1 - \lambda)|(T^D)^*|\big\|}{2}
\leq \lambda\|T\| + (1 - \lambda)\|T^D\|.
\end{align}
\endgroup
In particular,
\begingroup\makeatletter\def\f@size{10}\check@mathfonts
\begin{align}\label{T.4003.2}
\|\widehat{T}\| \leq \frac{\big\|\,|T| + |T^D|\big\|
+ \big\|\,|T^*| + |(T^D)^*|\big\|}{4} \leq \|T\|.
\end{align}
\endgroup
\end{theorem}
\begin{proof}
Let $x, y \in\mathcal{H}$. Since $M_{\lambda}(T) = \lambda T + (1-\lambda)T^D$, we have
\begin{align*}
|\langle M_{\lambda}(T)x, y\rangle| & \leq
\lambda|\langle Tx, y\rangle| + (1-\lambda)|\langle T^Dx, y\rangle|
\\& \leq \lambda{\langle |T|x, x\rangle}^{1/2} {\langle |T^*|y, y\rangle}^{1/2}
+ (1-\lambda){\langle |T^D|x, x\rangle}^{1/2}{\langle |(T^D)^*|y, y\rangle}^{1/2}
\\&\qquad \qquad \big(\mbox{by the mixed Schwarz inequality (see \cite[p. 75]{Halmos})}\big)
\\& \leq \lambda \frac{\langle |T|x, x\rangle + \langle |T^*|y, y\rangle}{2}
+ (1-\lambda)\frac{\langle |T^D|x, x\rangle + \langle |(T^D)^*|y, y\rangle}{2}
\\& \qquad \qquad \qquad \quad \big(\mbox{by the arithmetic-geometric mean inequality}\big)
\\& = \frac{\big\langle\big(\lambda|T|+(1-\lambda)|T^D|\big)x, x\big\rangle
+ \big\langle\big(\lambda|T^*|+(1-\lambda)|(T^D)^*|\big)y, y\big\rangle}{2}
\\& \leq \frac{\big\|\lambda|T| + (1 - \lambda)|T^D|\big\|
+ \big\|\lambda|T^*| + (1 - \lambda)|(T^D)^*|\big\|}{2}.
\end{align*}
Thus
\begin{align*}
|\langle M_{\lambda}(T)x, y\rangle|
\leq \frac{\big\|\lambda|T| + (1 - \lambda)|T^D|\big\|
+ \big\|\lambda|T^*| + (1 - \lambda)|(T^D)^*|\big\|}{2},
\end{align*}
whence
\begin{align*}
\big\|M_{\lambda}(T)\big\| \leq \frac{\big\|\lambda|T| + (1 - \lambda)|T^D|\big\|
+ \big\|\lambda|T^*| + (1 - \lambda)|(T^D)^*|\big\|}{2}.
\end{align*}
To prove the second inequality in (\ref{T.4003.1}),
since $\big\||A^*|\big\| = \big\||A|\big\| = \|A\|$ for any $A\in\mathbb{B}(\mathcal{H})$,
by the triangle inequality we have
\begingroup\makeatletter\def\f@size{10}\check@mathfonts
\begin{align*}
\big\|\lambda|T| + (1 - \lambda)|T^D|\big\|
+ \big\|\lambda|T^*| + (1 - \lambda)|(T^D)^*|\big\|
\leq  2\lambda\|T\| + 2(1 - \lambda)\|T^D\|.
\end{align*}
\endgroup
Finally, for $\lambda=1/2$, by (\ref{T.4003.1}) and Proposition \ref{P.1003} we obtain (\ref{T.4003.2}).
\end{proof}
The following example shows that the inequality (\ref{T.4003.2}) is a nontrivial improvement.
\begin{example}\label{E.5003}
Let $T = \begin{bmatrix}
0 & 1 \\
0 & 1
\end{bmatrix}$.
Simple computations show that
the canonical polar decomposition of $T$ is $T = U|T|$,
where
$U = \frac{\sqrt{2}}{2}\begin{bmatrix}
0 & 1 \\
0 & 1
\end{bmatrix}$
and
$|T| = \sqrt{2}\begin{bmatrix}
0 & 0 \\
0 & 1
\end{bmatrix}$.
Hence
$T^D = \begin{bmatrix}
0 & 0 \\
0 & 1
\end{bmatrix}$.
We have
\begin{align*}
\|\widehat{T}\| = \frac{1}{2}\|T + T^D\| = \frac{1}{2}\left\|\begin{bmatrix}
0 & 1 \\
0 & 2
\end{bmatrix}\right\| \simeq 1.1180,
\end{align*}
\begin{align*}
\big\|\,|T| + |T^D|\big\| =
\left\|\begin{bmatrix}
0 & 0 \\
0 & \sqrt{2}
\end{bmatrix} +
\begin{bmatrix}
0 & 0 \\
0 & 1
\end{bmatrix}\right\|
= \left\|\begin{bmatrix}
0 & 0 \\
0 & \sqrt{2} +1
\end{bmatrix} \right\| \simeq 2.4142
\end{align*}
and
\begingroup\makeatletter\def\f@size{10}\check@mathfonts
\begin{align*}
\big\|\,|T^*| + |(T^D)^*|\big\|=
\left\|\begin{bmatrix}
\sqrt{2}/2 & \sqrt{2}/2 \\
\sqrt{2}/2 & \sqrt{2}/2
\end{bmatrix} +
\begin{bmatrix}
0 & 0 \\
0 & 1
\end{bmatrix}\right\|
= \left\|\begin{bmatrix}
\sqrt{2}/2 & \sqrt{2}/2 \\
\sqrt{2}/2 & \sqrt{2}/2 +1
\end{bmatrix} \right\| \simeq 2.0731.
\end{align*}
\endgroup
Thus
\begingroup\makeatletter\def\f@size{8}\check@mathfonts
\begin{align*}
\|\widehat{T}\|\simeq1.1180 < \frac{\big\|\,|T| + |T^D|\big\|
+ \big\|\,|T^*| + |(T^D)^*|\big\|}{4} \simeq 1.1218 < \frac{\|T\| + \|T^D\|}{2} \simeq 1.2071 < \|T\| \simeq 1.4142.
\end{align*}
\endgroup
\end{example}
The following result may be stated as well.
\begin{theorem}\label{T.6003}
Let $T\in\mathbb{B}(\mathcal{H})$ and let $\lambda \in [0, 1]$. Then
\begingroup\makeatletter\def\f@size{8}\check@mathfonts
\begin{align*}
\big\|M_{\lambda}(T)\big\|^4
&\leq \|Q_\lambda(T)\|^2
+ 4(\lambda - \lambda^2)^2\omega^2(T^*T^D)
+ 2(\lambda-\lambda^2)\omega\big(Q_\lambda(T)T^*T^D + T^*T^DQ_\lambda(T)\big),
\end{align*}
\endgroup
where $Q_\lambda(T) = \lambda^2|T|^2 + (1 - \lambda)^2|T^D|^2$.
\end{theorem}
\begin{proof}
Notice first that, for every $A\in\mathbb{B}(\mathcal{H})$, it can be verified that
$\left\|\begin{bmatrix}
0 & A \\
A^* & 0
\end{bmatrix}\right\| = \|A\|$.
Also, by (\ref{1.2}), we have $\|A + A^*\| \leq 2\omega(A)$.
Therefore,
\begingroup\makeatletter\def\f@size{8}\check@mathfonts
\begin{align*}
\big\|M_{\lambda}(T)\big\|^4 &= \left\|\begin{bmatrix}
0 & M_{\lambda}(T) \\
M^*_{\lambda}(T) & 0
\end{bmatrix} \right\|^4
\\& = \left\|\begin{bmatrix}
0 & \lambda T\\
(1-\lambda)(T^D)^* & 0
\end{bmatrix}
+
\begin{bmatrix}
0 & (1-\lambda)T^D \\
\lambda T^* & 0
\end{bmatrix} \right\|^4
\\& = \left\|\begin{bmatrix}
0 & \lambda T\\
(1-\lambda)(T^D)^* & 0
\end{bmatrix}
+
{\begin{bmatrix}
0 & \lambda T \\
(1-\lambda)(T^D)^* & 0
\end{bmatrix}}^*\right\|^4
\\& \leq 16\,\omega^4\left(\begin{bmatrix}
0 & \lambda T\\
(1-\lambda)(T^D)^* & 0
\end{bmatrix}\right)
\\& = 16\,\displaystyle{\sup_{\theta \in \mathbb{R}}}
\left\|\mathfrak{Re}\left(e^{i\theta}\begin{bmatrix}
0 & \lambda T\\
(1-\lambda)(T^D)^* & 0
\end{bmatrix}\right)\right\|^4 \qquad \qquad \qquad \qquad \qquad \qquad \mbox{(by (\ref{1.2}))}
\\& = \displaystyle{\sup_{\theta \in \mathbb{R}}}
\left\|\begin{bmatrix}
0 & \lambda e^{i\theta}T + (1-\lambda)e^{-i\theta}T^D\\
\lambda e^{-i\theta}T^* + (1-\lambda)e^{i\theta}(T^D)^* & 0
\end{bmatrix}\right\|^4
\\& = \displaystyle{\sup_{\theta \in \mathbb{R}}}
\Big\|\lambda e^{i\theta}T + (1-\lambda)e^{-i\theta}T^D\Big\|^4
\\& = \displaystyle{\sup_{\theta \in \mathbb{R}}}
\Big\|\Big(\lambda e^{-i\theta}T^* + (1-\lambda)e^{i\theta}(T^D)^*\Big)^*
\Big(\lambda e^{i\theta}T + (1-\lambda)e^{-i\theta}T^D\Big)\Big\|^2
\\& = \displaystyle{\sup_{\theta \in \mathbb{R}}}
\Big\|\lambda^2|T|^2 + (1-\lambda)^2|T^D|^2 + 2(\lambda - \lambda^2)\mathfrak{Re}\big(e^{-2i\theta}T^*T^D\big)\Big\|^2
\\& = \displaystyle{\sup_{\theta \in \mathbb{R}}}
\Big\|\Big(Q_\lambda(T) + 2(\lambda - \lambda^2)\mathfrak{Re}\big(e^{-2i\theta}T^*T^D\big)\Big)^2\Big\|
\\& \leq \|Q_\lambda(T)\|^2 + 4(\lambda - \lambda^2)^2\displaystyle{\sup_{\theta \in \mathbb{R}}}\Big\|\mathfrak{Re}\big(e^{-2i\theta}T^*T^D\big)\Big\|^2
\\& \qquad \qquad \qquad \qquad + 2(\lambda - \lambda^2)\displaystyle{\sup_{\theta \in \mathbb{R}}}\Big\|\mathfrak{Re}\Big(e^{-2i\theta}\Big(Q_\lambda(T)T^*T^D + T^*T^DQ_\lambda(T)\Big)\Big)\Big\|
\\& = \|Q_\lambda(T)\|^2
+ 4(\lambda - \lambda^2)^2\omega^2(T^*T^D)
+ 2(\lambda-\lambda^2)\omega\big(Q_\lambda(T)T^*T^D + T^*T^DQ_\lambda(T)\big). \quad \mbox{(by (\ref{1.2}))}
\end{align*}
\endgroup
\end{proof}
\begin{remark}\label{R.7003}
In \cite{Fo.Ho}, it has been shown that if $A, B\in\mathbb{B}(\mathcal{H})$, then
\begin{align*}
\omega(AB + BA^*) \leq 2\|A\|\omega(B).
\end{align*}
So, we see that
\begingroup\makeatletter\def\f@size{8}\check@mathfonts
\begin{align*}
\|Q_\lambda(T)\|^2
&+ 4(\lambda - \lambda^2)^2\omega^2(T^*T^D)
+ 2(\lambda-\lambda^2)\omega\big(Q_\lambda(T)T^*T^D + T^*T^DQ_\lambda(T)\big)
\\& \leq \|Q_\lambda(T)\|^2
+ 4(\lambda - \lambda^2)^2\omega^2(T^*T^D)
+ 4(\lambda-\lambda^2)\|Q_\lambda(T)\|\omega\big(T^*T^D\big)
\\& \leq \|Q_\lambda(T)\|^2
+ 4\lambda(1 - \lambda)\big\|T^*T^D\big\|\Big(\lambda(1-\lambda)\big\|T^*T^D\big\| + \|Q_\lambda(T)\|\Big) \qquad \quad \Big(\mbox{by (\ref{1.1})}\Big)
\\& \leq \lambda^4\|T\|^4 + 2\lambda^2(1-\lambda)^2\|T\|^2\|T^D\|^2 + (1-\lambda)^4\|T^D\|^4
\\& \qquad \qquad \qquad + 4\lambda(1 - \lambda)\|T\|\|T^D\|\Big(\lambda(1-\lambda)\|T\|\|T^D\| + \lambda^2\|T\|^2 + (1-\lambda)^2\|T^D\|^2\Big)
\\& = \big(\lambda\|T\| + (1-\lambda)\|T^D\|\big)^4.
\end{align*}
\endgroup
Hence, Theorem \ref{T.6003} is an improvement of the second inequality in (\ref{T.2003.1}).
\end{remark}
In the following theorem, we give a necessary and sufficient condition for the
equality $\big\|M_{\lambda}(T)\big\| = \|T\|$.
\begin{theorem}\label{T.8003}
Let $T\in\mathbb{B}(\mathcal{H})$ and let $\lambda \in (0, 1)$.
Then the following statements are equivalent:
\begin{itemize}
\item[(i)] $\big\|M_{\lambda}(T)\big\| = \|T\|$.
\item[(ii)] There exists a sequence of unit vectors $\{x_n\}$ in $\mathcal{H}$ such that
\begin{align*}
\displaystyle{\lim_{n\rightarrow +\infty}}\langle Tx_n, T^Dx_n\rangle = \|T\|^2.
\end{align*}
\end{itemize}
\end{theorem}
\begin{proof}
We may and shall assume that $T^D\neq 0$ otherwise the equivalence
(i)$\Leftrightarrow$(ii) trivially holds.\\
(i)$\Rightarrow$(ii) Let $\big\|M_{\lambda}(T)\big\| = \|T\|$.
Since
\begin{align*}
\big\|M_{\lambda}(T)\big\| = \displaystyle{\sup_{\|x\| = 1}}\Big\|\big(\lambda T + (1-\lambda)T^D\big)x\Big\|,
\end{align*}
so there exists a sequence of unit vectors $\{x_n\}$ in $\mathcal{H}$ such that
\begin{align}\label{I.1.T.4}
\displaystyle{\lim_{n\rightarrow +\infty}}\Big\|\lambda Tx_n + (1-\lambda)T^Dx_n\Big\| = \|T\|.
\end{align}
For every $n\in \mathbb{N}$, we have
\begin{align*}
\Big\|\lambda Tx_n + (1-\lambda)T^Dx_n\Big\|&\leq \lambda\|Tx_n\| + (1-\lambda)\|T^Dx_n\|
\\&\leq \lambda\|Tx_n\| + (1-\lambda)\|T^D\|
\\& \leq \lambda\|Tx_n\| + (1-\lambda)\|T\| \qquad \qquad \big(\mbox{by Proposition \ref{P.1003}}\big)
\\& \leq \lambda\|T\| + (1-\lambda)\|T\| = \|T\|,
\end{align*}
whence
\begin{align}\label{I.2.T.4.1}
\displaystyle{\lim_{n\rightarrow +\infty}}\|Tx_n\| = \|T\|.
\end{align}
Similarly, we obtain
\begin{align}\label{I.2.T.4.2}
\displaystyle{\lim_{n\rightarrow +\infty}}{\|T^Dx_n\|}_{A} = \|T\|.
\end{align}
Since
\begingroup\makeatletter\def\f@size{10}\check@mathfonts
\begin{align*}
\Big\|\lambda Tx_n + (1-\lambda)T^Dx_n\Big\|^2 = \lambda^2\|Tx_n\|^2
+ 2(\lambda-\lambda^2)\mathfrak{Re}\langle Tx_n, T^Dx_n\rangle + (1-\lambda)^2\|T^Dx_n\|^2
\end{align*}
\endgroup
for every $n\in \mathbb{N}$, from (\ref{I.1.T.4}), (\ref{I.2.T.4.1}), and (\ref{I.2.T.4.2}), we obtain
\begin{align}\label{I.3.T.4}
\displaystyle{\lim_{n\rightarrow +\infty}}\mathfrak{Re}\langle Tx_n, T^Dx_n\rangle = \|T\|^2.
\end{align}
In addition, by proposition \ref{P.1003}, we have
\begingroup\makeatletter\def\f@size{10}\check@mathfonts
\begin{align*}
\mathfrak{Re}^2\langle Tx_n, T^Dx_n\rangle + \mathfrak{Im}^2\langle Tx_n, T^Dx_n\rangle
= |\langle Tx_n, T^Dx_n\rangle|^2 \leq \|T\|\,\|T^D\| \leq \|T\|^2,
\end{align*}
\endgroup
for every $n\in \mathbb{N}$.
Hence, by (\ref{I.3.T.4}), we conclude that
$\displaystyle{\lim_{n\rightarrow +\infty}}\mathfrak{Im}\langle Tx_n, T^Dx_n\rangle = 0$.
Now (\ref{I.3.T.4}) implies that
$\displaystyle{\lim_{n\rightarrow +\infty}}\langle Tx_n, T^Dx_n\rangle = \|T\|^2$.

(ii)$\Rightarrow$(i) Suppose that there exists a sequence of unit vectors $\{x_n\}$ in $\mathcal{H}$ such that
$\displaystyle{\lim_{n\rightarrow +\infty}}\langle Tx_n, T^Dx_n\rangle = \|T\|^2$.
Hence $\displaystyle{\lim_{n\rightarrow +\infty}}\mathfrak{Re}\langle Tx_n, T^Dx_n\rangle = \|T\|^2$.
Since
\begin{align*}
|\langle Tx_n, T^Dx_n\rangle|
\leq \|Tx_n\|\,\|T^Dx_n\|\leq \|Tx_n\|\,\|T^D\| \leq \|Tx_n\|\,\|T\| \leq \|T\|^2
\end{align*}
for every $n\in \mathbb{N}$, we obtain $\displaystyle{\lim_{n\rightarrow +\infty}}\|Tx_n\| = \|T\|$ and $\|T^D\| = \|T\|$.
Also, by a similar argument, we get
$\displaystyle{\lim_{n\rightarrow +\infty}}\|T^Dx_n\| = \|T\|$.
Therefore, by (\ref{T.2003.2}) it follows that
\begingroup\makeatletter\def\f@size{10}\check@mathfonts
\begin{align*}
\|T\|^2&= \lambda^2\displaystyle{\lim_{n\rightarrow +\infty}}\|Tx_n\|^2
+ 2(\lambda-\lambda^2)\displaystyle{\lim_{n\rightarrow +\infty}}\mathfrak{Re}\langle Tx_n, T^Dx_n\rangle
+ (1-\lambda)^2\displaystyle{\lim_{n\rightarrow +\infty}}\|T^Dx_n\|^2
\\& = \displaystyle{\lim_{n\rightarrow +\infty}}\Big\|M_{\lambda}(T)x_n\Big\|^2 \leq \Big\|M_{\lambda}(T)\Big\|^2 \leq \|T\|^2,
\end{align*}
\endgroup
which gives $\big\|M_{\lambda}(T)\big\| = \|T\|$.
\end{proof}
If $T\in\mathbb{B}(\mathcal{H})$, then
$\big\|M_{\lambda}(T)\big\| \leq \lambda\|T\| + (1-\lambda)\|T^D\|$
and hence
\begin{align}\label{T.11003.01}
\big\|M_{\lambda}(T)\big\| \leq 2\max\{\lambda\|T\|, (1-\lambda)\|T^D\|\}.
\end{align}
In the following, we mimic \cite[Theorem 2.5]{H.K.S.2} to prove a condition for the equality in (\ref{T.11003.01}).
To do this, we need the following two lemmas.
The first lemma is well-known and can be found in \cite[Theorem 2.1]{B.B}.
\begin{lemma}\label{L.10003}
Let $A, B\in\mathbb{B}(\mathcal{H})$. Then the equation
$\|A + B\| = \|A\| + \|B\|$ holds if and only if $\|A\|\|B\| \in \overline{W\big(A^*B\big)}$.
\end{lemma}
The second lemma is interesting on its own right.
\begin{lemma}\label{L.9003}
Let $T\in\mathbb{B}(\mathcal{H})$ and let $\lambda \in (0, 1)$.
Then the following conditions are equivalent:
\begin{itemize}
\item[(i)] $\|\widehat{T}\| = \frac{1}{2}\big(\|T\| + \|T^D\|\big)$.
\item[(ii)] $\|M_{\lambda}(T)\| = \lambda\|T\| + (1-\lambda)\|T^D\|$.
\end{itemize}
\end{lemma}
\begin{proof}
It is not hard to see that the equality case in the triangle
inequality for two elements $x$ and $y$ in a normed linear space $\mathcal{X}$ is preserved for
their positive multiples $\alpha x$ ($\alpha\geq 0$) and $\beta y$ ($\beta\geq0$).
Therefore, the proof follows immediately from the above property.
\end{proof}
\begin{theorem}\label{T.11003}
Let $T\in\mathbb{B}(\mathcal{H})$ and let $\lambda \in (0, 1)$.
Then the following statements are equivalent:
\begin{itemize}
\item[(i)] $\|M_{\lambda}(T)\| = 2\max\{\lambda\|T\|, (1-\lambda)\|T^D\|\}$.
\item[(ii)] $\omega(T^*T^D) = \frac{1}{\lambda - \lambda^2}\max\{\lambda^2\|T\|^2, (1-\lambda)^2\|T^D\|^2\}$.
\end{itemize}
\end{theorem}
\begin{proof}
(i)$\Rightarrow$(ii) Let $\|M_{\lambda}(T)\| = 2\max\{\lambda\|T\|, (1-\lambda)\|T^D\|\}$. Since
\begin{align*}
\|M_{\lambda}(T)\| \leq \lambda\|T\| + (1-\lambda)\|T^D\| \leq 2\max\{\lambda\|T\|, (1-\lambda)\|T^D\|\},
\end{align*}
we get
$\lambda\|T\| + (1-\lambda)\|T^D\| = 2\max\{\lambda\|T\|, (1-\lambda)\|T^D\|\}$.
Hence
\begin{align}\label{T.11003.1}
\lambda\|T\| = (1-\lambda)\|T^D\| = \max\{\lambda\|T\|, (1-\lambda)\|T^D\|\}.
\end{align}
Therefore,
$\|M_{\lambda}(T)\| = \lambda\|T\| + (1-\lambda)\|T^D\|$,
and by Lemma \ref{L.9003}, we obtain $\|\widehat{T}\| = \frac{1}{2}\big(\|T\| + \|T^D\|\big)$,
or equivalently, $\|T + T^D\| = \|T\| + \|T^D\|$.
By Lemma \ref{L.10003} it follows that $\|T\|\,\|T^D\| \in \overline{W\big(T^*T^D\big)}$
and so $\|T\|\,\|T^D\| \leq \omega\big(T^*T^D\big)$.
Now (\ref{T.11003.1}) implies that
\begin{align}\label{I.1.T.5}
\frac{1}{\lambda - \lambda^2}\max\{\lambda^2\|T\|^2, (1-\lambda)^2\|T^D\|^2\} = \|T\|\,\|T^D\| \leq \omega\big(T^*T^D\big).
\end{align}
On the other hand, by the second inequality in (\ref{1.1}) and the arithmetic-geometric mean inequality, we have
\begingroup\makeatletter\def\f@size{10}\check@mathfonts
\begin{align}\label{I.3.T.4.5}
\omega\big(T^*T^D\big) &\leq \big\|T^*T^D\big\| \leq \|T^*\|\,\|T^D\|
= \frac{\|\lambda T\|\,\|(1-\lambda)T^D\|}{\lambda - \lambda^2}\nonumber
\\& \leq \frac{\lambda^2\|T\|^2 + (1-\lambda)^2\|T^D\|^2}{2(\lambda - \lambda^2)}
\leq \frac{\max\{\lambda^2\|T\|^2, (1-\lambda)^2\|T^D\|^2\}}{\lambda - \lambda^2}.
\end{align}
\endgroup
Thus
\begin{align}\label{I.3.T.5}
\omega(T^*T^D) \leq \frac{1}{\lambda - \lambda^2}\max\{\lambda^2\|T\|^2, (1-\lambda)^2\|T^D\|^2\}.
\end{align}
By (\ref{I.1.T.5}) and (\ref{I.3.T.5}), we conclude that
\begin{align*}
\omega(T^*T^D) = \frac{1}{\lambda - \lambda^2}\max\{\lambda^2\|T\|^2, (1-\lambda)^2\|T^D\|^2\}.
\end{align*}

(ii)$\Rightarrow$(i) Let $\omega(T^*T^D) = \frac{1}{\lambda - \lambda^2}\max\{\lambda^2\|T\|^2, (1-\lambda)^2\|T^D\|^2\}$.
From (\ref{I.3.T.4.5}) it follows that $\lambda\|T\| = (1-\lambda)\|T^D\|$ and
$\omega(T^*T^D) = \|T\|\,\|T^D\|$. Thus
$\|T\|\,\|T^D\|\in\overline{W\big(T^*T^D\big)}$.
Now, by Lemma \ref{L.10003}, we obtain
$\|T + T^D\| = \|T\| + \|T^D\|$, or equivalently, $\|\widehat{T}\| = \frac{1}{2}\big(\|T\| + \|T^D\|\big)$.
Hence by Lemma \ref{L.9003} we get $\|M_{\lambda}(T)\| = \lambda\|T\| + (1-\lambda)\|T^D\|$
and so $\|M_{\lambda}(T)\| = 2\max\{\lambda\|T\|, (1-\lambda)\|T^D\|\}$.
\end{proof}
As a consequence of Theorem \ref{T.11003} and Proposition \ref{P.1003}, we have the following result.
\begin{corollary}\label{C.12003}
Let $T\in\mathbb{B}(\mathcal{H})$. Then the following statements are equivalent:
\begin{itemize}
\item[(i)] $\|\widehat{T}\| = \|T\|$.
\item[(ii)] $\omega(T^*T^D) = \|T\|^2$.
\end{itemize}
\end{corollary}
\section{Bounds for the numerical radius of the $\lambda$-mean transform}
In this section, inspired by the numerical radius inequalities of bounded linear operators
in \cite{A.K.1, H.K.S.2, M.X.Z}, we state several numerical radius
inequalities for the $\lambda$-mean transform of Hilbert space operators.
Our first result reads as follows.
\begin{theorem}\label{T.1004}
Let $T\in\mathbb{B}(\mathcal{H})$ and let $\lambda \in [0, 1]$.
Then
\begin{align}\label{T.1004.1}
2\sqrt{\lambda - \lambda^2}\,\omega(\widetilde{T})\leq \omega\big(M_{\lambda}(T)\big) \leq \lambda\omega(T) + (1-\lambda)\omega(T^D).
\end{align}
In particular,
\begin{align}\label{T.1004.2}
2\sqrt{\lambda - \lambda^2}\,r(T)\leq \omega\big(M_{\lambda}(T)\big) \leq \omega(T).
\end{align}
\end{theorem}
\begin{proof}
Let $T=U|T|$ be the canonical polar decomposition of $T$.
It has been indicated in \cite{S.M.Y} that, for positive operators $A, B \in\mathbb{B}(\mathcal{H})$
and any $X\in\mathbb{B}(\mathcal{H})$,
\begin{align}\label{T.1004.3}
\omega\big(A^{1/2}XB^{1/2}\big)\leq \omega\left(\frac{AX + XB}{2}\right).
\end{align}
By letting $A=(1-\lambda)|T|, B= \lambda |T|$ and $X = U$ in the inequality (\ref{T.1004.3}),
it follows that
\begin{align*}
\omega\Big(\sqrt{1-\lambda}|T|^{1/2}U\sqrt{\lambda}|T|^{1/2}\Big)\leq \omega\left(\frac{(1-\lambda)|T|U + U\lambda|T|}{2}\right).
\end{align*}
Therefore,
\begin{align*}
\sqrt{\lambda - \lambda^2}\omega\big(|T|^{1/2}U|T|^{1/2}\big)\leq \frac{1}{2}\omega\big((1-\lambda)T^D + \lambda T\big),
\end{align*}
and hence
\begin{align}\label{4.P.1.004}
2\sqrt{\lambda - \lambda^2}\omega(\widetilde{T})\leq \omega\big(M_{\lambda}(T)\big).
\end{align}
In addition, since $M_{\lambda}(T)=\lambda T + (1-\lambda)T^D$ and $\omega(\cdot)$ is a norm on $\mathbb{B}(\mathcal{H})$, we have
\begin{align}\label{4.P.1.005}
\omega\big(M_{\lambda}(T)\big) \leq \lambda\omega(T) + (1-\lambda)\omega(T^D).
\end{align}
Now, the inequalities (\ref{T.1004.1}) follow from the inequalities (\ref{4.P.1.004}) and (\ref{4.P.1.005}).

Also, for any $A\in\mathbb{B}(\mathcal{H})$, it is known (see \cite{St.Wi}) that
\begin{align*}
\overline{W(A)} = \displaystyle{\bigcap_{\gamma \in\mathbb{C}}}\big\{\xi: \, |\xi - \gamma| \leq \|A - \gamma I\|\big\}.
\end{align*}
So, by Proposition \ref{P.1003}, we have
\begin{align*}
\overline{W(T^D)} &= \displaystyle{\bigcap_{\gamma \in\mathbb{C}}}\big\{\xi: \, |\xi - \gamma| \leq \|T^D - \gamma I\|\big\}
\\& \subseteq \displaystyle{\bigcap_{\gamma \in\mathbb{C}}}\big\{\xi: \, |\xi - \gamma| \leq \|T - \gamma I\|\big\} = \overline{W(T)},
\end{align*}
whence $\overline{W(T^D)} \subseteq \overline{W(T)}$.
Thus $\omega(T^D) \leq \omega(T)$.
Further, by $r(T) = r(\widetilde{T})$ and the first inequality in (\ref{1.1}), we get
$r(T)\leq \omega(\widetilde{T})$.
Therefore, we deduce the inequalities (\ref{T.1004.2}).
\end{proof}
\begin{remark}\label{R.2004}
Let $T\in\mathbb{B}(\mathcal{H})$.
If $\omega(T) = 0$, then by (\ref{T.1004.2}) we have $\omega\big(M_{\lambda}(T)\big) = 0$.
Now, let $\omega\big(M_{\lambda}(T)\big) = 0$ and let $\lambda \in (0, 1]$. Since, by (\ref{1.1}), we have
$\frac{1}{2}\big\|M_{\lambda}(T)\big\|\leq \omega\big(M_{\lambda}(T)\big)$,
therefore we reach that $\big\|M_{\lambda}(T)\big\|= 0$.
So, by Proposition \ref{P.2003}, it follows that $\|T\| = 0$.
Again by (\ref{1.1}), we obtain $\omega(T) = 0$.
\end{remark}
\begin{remark}\label{R.3004}
Let $T\in\mathbb{B}(\mathcal{H})$ and let $\lambda \in [0, 1]$.
By the triangle inequality and Proposition \ref{P.1003}, we have
\begin{align*}
\overline{W\big(M_{\lambda}(T)\big)}
&= \displaystyle{\bigcap_{\gamma \in\mathbb{C}}}\Big\{\xi: \, |\xi - \gamma| \leq \|M_{\lambda}(T) - \gamma I\|\Big\}
\\& = \displaystyle{\bigcap_{\gamma \in\mathbb{C}}}\Big\{\xi: \, |\xi - \gamma|
\leq \Big\|\lambda (T -\gamma I) + (1-\lambda)(T^D - \gamma I)\Big\|\Big\}
\\& \subseteq \displaystyle{\bigcap_{\gamma \in\mathbb{C}}}\Big\{\xi: \, |\xi - \gamma|
\leq \lambda \|T - \gamma I\| + (1-\lambda)\|T^D - \gamma I\|\Big\}
\\& \subseteq \displaystyle{\bigcap_{\gamma \in\mathbb{C}}}\Big\{\xi: \, |\xi - \gamma|
\leq \lambda \|T -\gamma I\| + (1-\lambda)\|T - \gamma I\|\Big\}
\\& = \displaystyle{\bigcap_{\gamma \in\mathbb{C}}}\Big\{\xi: \, |\xi - \gamma|
\leq \|T -\gamma I\|\Big\}
= \overline{W(T)},
\end{align*}
and hence
\begin{align*}
\overline{W\big(M_{\lambda}(T)\big)} \subseteq \overline{W(T)}.
\end{align*}
\end{remark}
\begin{remark}\label{R.4004}
Let $T\in\mathbb{B}(\mathcal{H})$ and let $\lambda \in (0, 1]$.
For $\zeta \in \mathbb{C}$, if $T = \zeta I$, then it is obvious that
$W\big(M_{\lambda}(T)\big) = \{\zeta\}$.
Let us now suppose $W\big(M_{\lambda}(T)\big) = \{\zeta\}$ and let $T=U|T|$ be the canonical polar decomposition of $T$.
For $A\in\mathbb{B}(\mathcal{H})$,
it is well known that $W(A) = \{\zeta\}$ if and only if $A = \zeta I$ (see \cite{Halmos}).
Therefore, $M_{\lambda}(T) = \zeta I$, or equivalently, $\lambda U|T| + (1-\lambda)|T|U = \zeta I$.
From this it follows that 
\begin{align*}
\zeta U^* = \lambda U^*U|T| + (1-\lambda)U^*|T|U = \lambda|T| + (1-\lambda)U^*|T|U.
\end{align*}
Taking $*$-operation yields
$\overline{\zeta} U = \lambda |T| + (1-\lambda)U^*|T|U$. Hence $\overline{\zeta} U$ is an orthogonal projection.
On the other hand, since $\mathcal{N}(T) = \mathcal{N}(U)$, we have $\mathcal{N}(T) \subseteq \mathcal{N}\big(M_{\lambda}(T)\big) = \{0\}$
and so $T$ is one-to-one. Thus $\overline{\zeta} U$ is an isometry, so
$\overline{\zeta} U = I$. Therefore,
$\overline{\zeta}T = \overline{\zeta}U|T| = |T|$
and $|T| = |\zeta|^2 I$. This implies $T = \zeta I$.
\end{remark}
In the following result we obtain an improvement of the second inequality in (\ref{T.1004.1}).
\begin{theorem}\label{T.5004}
Let $T\in\mathbb{B}(\mathcal{H})$ and let $\lambda \in [0, 1]$. Then
\begingroup\makeatletter\def\f@size{8}\check@mathfonts
\begin{align*}
\omega\big(M_{\lambda}(T)\big) \leq 2\int_0^1\omega\Big(\lambda t T + (1-\lambda)(1-t)T^D\Big)dt
\leq \frac{1}{2}\Big(\lambda\omega(T) + (1 - \lambda)\omega(T^D) + \omega\big(M_{\lambda}(T)\big)\Big).
\end{align*}
\endgroup
In particular,
\begingroup\makeatletter\def\f@size{10}\check@mathfonts
\begin{align*}
\omega(\widehat{T})\leq \int_0^1\omega\big(M_{t}(T)\big)dt
\leq \frac{1}{4}\Big(\omega(T) + \omega(T^D) + 2\omega(\widehat{T})\Big)
\leq \omega(T).
\end{align*}
\endgroup
\end{theorem}
\begin{proof}
Let $f(t) : = \omega\big(t\lambda T + (1 - t)(1-\lambda)T^D\big)$ for $t\in[0, 1]$.
It is easy to see that the function $f$ is convex,
and so we have
\begin{align*}
f\left(\frac{0+1}{2}\right) \leq \frac{1}{1-0}\int_0^1 f(t)dt.
\end{align*}
Therefore
\begin{align*}
\omega\left(\frac{1}{2}\lambda T + \frac{1}{2}(1-\lambda)T^D\right)
\leq \int_0^1 \omega\Big(t\lambda T + (1 - t)(1-\lambda)T^D\Big)dt,
\end{align*}
and hence
\begin{align}\label{T.5004.1}
\omega\big(M_{\lambda}(T)\big) \leq 2\int_0^1\omega\Big(\lambda t T + (1-\lambda)(1-t)T^D\Big)dt.
\end{align}
On the other hand, by the Hammer-Bullen inequality (see, e.g., \cite{Ni.Pe}), we have
\begin{align*}
\frac{2}{1-0}\int_0^1 f(t)dt \leq \frac{f(0) + f(1)}{2} + f\left(\frac{0+1}{2}\right).
\end{align*}
Thus
\begingroup\makeatletter\def\f@size{8}\check@mathfonts
\begin{align*}
2\int_0^1 \omega\Big(t\lambda T + (1 - t)(1-\lambda)T^D\Big)dt
\leq \frac{(1-\lambda)\omega(T^D) + \lambda \omega(T)}{2} + \omega\left(\frac{1}{2}\lambda T + \frac{1}{2}(1-\lambda)T^D\right),
\end{align*}
\endgroup
whence
\begingroup\makeatletter\def\f@size{10}\check@mathfonts
\begin{align}\label{T.5004.2}
2\int_0^1\omega\Big(\lambda t T + (1-\lambda)(1-t)T^D\Big)dt
\leq \frac{1}{2}\Big(\lambda\omega(T) + (1 - \lambda)\omega(T^D) + \omega\big(M_{\lambda}(T)\big)\Big).
\end{align}
\endgroup
From (\ref{T.5004.1}) and (\ref{T.5004.2}), we deduce the desired result.
\end{proof}
\begin{remark}\label{R.6004}
Let $T = \begin{bmatrix}
0 & 1 \\
0 & 1
\end{bmatrix}$.
By Example \ref{E.5003}, we have
$T^D = \begin{bmatrix}
0 & 0 \\
0 & 1
\end{bmatrix}$. Hence
$\widehat{T} = \begin{bmatrix}
0 & 1/2 \\
0 & 1
\end{bmatrix}$.
It is easy to see that
$\omega(T) = \frac{1+\sqrt{2}}{2}, \omega(T^D) = 1, \omega(\widehat{T}) = \frac{2+\sqrt{5}}{4}$,
and
\begin{align*}
\int_0^1 \omega\big(M_{t}(T)\big)dt = \int_0^1 \frac{1+\sqrt{1+t^2}}{2}dt \simeq 1.0739.
\end{align*}
Thus
\begingroup\makeatletter\def\f@size{10}\check@mathfonts
\begin{align*}
\omega(\widehat{T})\simeq 1.0590 &< \int_0^1\omega\big(M_{t}(T)\big)dt \simeq 1.0739
\\&< \frac{1}{4}\Big(\omega(T) + \omega(T^D) + 2\omega(\widehat{T})\Big) \simeq 1.0812
\\&< \frac{1}{2}\big(\omega(T) + \omega(T^D)\big) \simeq 1.1035 < \omega(T)\simeq 1.2071.
\end{align*}
\endgroup
Therefore, the inequalities in Theorem \ref{T.5004} are nontrivial improvements.
\end{remark}
Next, we present another improvement of the second inequality in (\ref{T.1004.1}).
\begin{theorem}\label{T.7004}
Let $T\in\mathbb{B}(\mathcal{H})$ and let $\lambda \in [0, 1]$. Then
\begingroup\makeatletter\def\f@size{10}\check@mathfonts
\begin{align*}
\omega\big(M_{\lambda}(T)\big) &\leq \frac{1}{2} \big(\lambda\omega(T) + (1 - \lambda)\omega(T^D)\big)
\\& \qquad + \frac{1}{2}\sqrt{\big(\lambda\omega(T) - (1 - \lambda)\omega(T^D)\big)^2 +
4(\lambda - \lambda^2) \sup_{\theta \in \mathbb{R}}\big\|\mathfrak{Re}(e^{i\theta}T)\mathfrak{Re}(e^{i\theta} T^D)\big\|}
\\& \leq \lambda\omega(T) + (1 - \lambda)\omega(T^D).
\end{align*}
\endgroup
In particular,
\begingroup\makeatletter\def\f@size{8}\check@mathfonts
\begin{align*}
\omega(\widehat{T}) \leq \frac{\omega(T) + \omega(T^D)
+ \sqrt{\big(\omega(T) - \omega(T^D)\big)^2 +
4\sup_{\theta \in \mathbb{R}}\big\|\mathfrak{Re}(e^{i\theta}T)\mathfrak{Re}(e^{i\theta} T^D)\big\|}}{4}
\leq \omega(T).
\end{align*}
\endgroup
\end{theorem}
\begin{proof}
For every $A, B\in\mathbb{B}(\mathcal{H})$, it is proved in \cite[Theorem 3.4]{A.K.1} that
\begingroup\makeatletter\def\f@size{10}\check@mathfonts
\begin{align*}
\omega(A + B) \leq \frac{1}{2} \big(\omega(A) + \omega(B)\big)
+ \frac{1}{2}\sqrt{\big(\omega(A) - \omega(B)\big)^2 +
4\sup_{\theta \in \mathbb{R}}\big\|\mathfrak{Re}(e^{i\theta}A)\mathfrak{Re}(e^{i\theta} B)\big\|}
\end{align*}
\endgroup
Applying the above inequality with $A = \lambda T$ and $B = (1-\lambda)T^D$, we get
\begingroup\makeatletter\def\f@size{10}\check@mathfonts
\begin{align*}
\omega\big(M_{\lambda}(T)\big) &\leq \frac{1}{2} \big(\lambda\omega(T) + (1 - \lambda)\omega(T^D)\big)
\\& \qquad + \frac{1}{2}\sqrt{\big(\lambda\omega(T) - (1 - \lambda)\omega(T^D)\big)^2 +
4(\lambda - \lambda^2) \sup_{\theta \in \mathbb{R}}\big\|\mathfrak{Re}(e^{i\theta}T)\mathfrak{Re}(e^{i\theta} T^D)\big\|}.
\end{align*}
\endgroup
Also, by (\ref{1.2}), we have
\begingroup\makeatletter\def\f@size{10}\check@mathfonts
\begin{align*}
\frac{1}{2}&\sqrt{\big(\lambda\omega(T) - (1 - \lambda)\omega(T^D)\big)^2 +
4(\lambda - \lambda^2) \sup_{\theta \in \mathbb{R}}\big\|\mathfrak{Re}(e^{i\theta}T)\mathfrak{Re}(e^{i\theta} T^D)\big\|}
\\& \qquad \leq \frac{1}{2}\sqrt{\big(\lambda\omega(T) - (1 - \lambda)\omega(T^D)\big)^2 +
4(\lambda - \lambda^2) \sup_{\theta \in \mathbb{R}}\big\|\mathfrak{Re}(e^{i\theta}T)\big\|
\sup_{\theta \in \mathbb{R}}\big\|\mathfrak{Re}(e^{i\theta} T^D)\big\|}
\\& \qquad \leq \frac{1}{2}\sqrt{\big(\lambda\omega(T) - (1 - \lambda)\omega(T^D)\big)^2 +
4(\lambda - \lambda^2)\omega(T)\omega(T^D)}
\\& \qquad = \frac{1}{2} \big(\lambda\omega(T) + (1 - \lambda)\omega(T^D)\big).
\end{align*}
\endgroup
Now, we deduce the desired result.
\end{proof}
\begin{remark}\label{R.8004}
The inequality obtained by us in Theorem \ref{T.7004} is a nontrivial improvement.
Consider $T = \begin{bmatrix}
0 & 1 \\
0 & 1
\end{bmatrix}$. By Remark \ref{R.6004}, we have
$T^D = \begin{bmatrix}
0 & 0 \\
0 & 1
\end{bmatrix}$,
$\omega(T) = \frac{1+\sqrt{2}}{2}, \omega(T^D) = 1$ and $\omega(\widehat{T}) = \frac{2+\sqrt{5}}{4}$.
Also, it is easy to check that
\begin{align*}
\sup_{\theta \in \mathbb{R}}\big\|\mathfrak{Re}(e^{i\theta}T)\mathfrak{Re}(e^{i\theta} T^D)\big\|
&= \frac{1}{4}\left\|\begin{bmatrix}
0 & 1+e^{2i\theta} \\
0 & 2+e^{2i\theta}+e^{-2i\theta}
\end{bmatrix}\right\|
\\& = \frac{1}{4}\sup_{\theta \in \mathbb{R}}\sqrt{6+10\cos(2\theta) + 4\cos^2(2\theta)}
= \frac{\sqrt{5}}{2}.
\end{align*}
Therefore,
\begingroup\makeatletter\def\f@size{8}\check@mathfonts
\begin{align*}
\omega(\widehat{T})\simeq 1.0590 &\leq \frac{\omega(T) + \omega(T^D)
+ \sqrt{\big(\omega(T) - \omega(T^D)\big)^2 +
4\sup_{\theta \in \mathbb{R}}\big\|\mathfrak{Re}(e^{i\theta}T)\mathfrak{Re}(e^{i\theta} T^D)\big\|}}{4}\simeq 1.0829
\\&< \frac{1}{2}\big(\omega(T) + \omega(T^D)\big) \simeq 1.1035 < \omega(T)\simeq 1.2071.
\end{align*}
\endgroup
\end{remark}
The following result may be stated as well.
\begin{proposition}\label{P.9004}
Let $T\in\mathbb{B}(\mathcal{H})$ and let $\lambda \in [0, 1]$.
For every unitary operator $V\in\mathbb{B}(\mathcal{H})$ and every $n\in\mathbb{Z}_{+}$,
\begin{align*}
\omega\Big(M^{(n)}_{\lambda}(VTV^*)\Big) = \omega\big(M^{(n)}_{\lambda}(T)\big).
\end{align*}
\end{proposition}
\begin{proof}
Let $T=U|T|$ be the canonical polar decomposition of $T$. Put $U_{0} = VUV^*$.
Since $V$ is a unitary operator and $U$ is a partial isometry, we have
\begin{align*}
U_{0}U^*_{0}U_{0} = (VUV^*)(VU^*V^*)(VUV^*) = VUU^*UV^* = VUV^* = U_{0}.
\end{align*}
Thus $U_{0}$ is a partial isometry.
Now, let $x\in \mathcal{N}(U_{0})$. Hence $VUV^*x = 0$. Since $V$ is a unitary operator,
we obtain $UV^*x = 0$, or equivalently, $V^*x\in \mathcal{N}(U) = \mathcal{N}(T)$.
Thus $TV^*x=0$ and so $x\in \mathcal{N}(VTV^*)$.
This implies $\mathcal{N}(U_{0})\subseteq\mathcal{N}(VTV^*)$.
Similarity, we have $\mathcal{N}(VTV^*)\subseteq\mathcal{N}(U_{0})$.
Thus $\mathcal{N}(U_{0})=\mathcal{N}(VTV^*)$.
Furthermore,
\begin{align*}
VTV^* = VU|T|V^* = VUV^*V|T|V^* = U_{0}(V|T|V^*) = U_{0}|VTV^*|.
\end{align*}
Therefore $U_{0}|VTV^*|$ is the polar decomposition of $VTV^*$.
So, we get
\begin{align*}
M_{\lambda}(VTV^*) &= \lambda (VTV^*) + (1-\lambda)(VTV^*)^D
\\&= \lambda U_{0}|VTV^*| + (1-\lambda)|VTV^*|U_{0}
\\&= \lambda VUV^*V|T|V^* + (1-\lambda)V|T|V^*VUV^*
\\&= V\Big(\lambda U|T| + (1-\lambda)|T|U\Big)V^*
= VM_{\lambda}(T)V^*,
\end{align*}
and hence $M_{\lambda}(VTV^*) = VM_{\lambda}(T)V^*$.
Now, by induction on $n$ the equality
\begin{align}\label{P.3004.1}
M^{(n)}_{\lambda}(VTV^*)=VM^{(n)}_{\lambda}(T)V^*,
\end{align}
holds immediately.
Since $\omega(\cdot)$ is a weakly unitarily invariant norm on $\mathbb{B}(\mathcal{H})$,
by (\ref{P.3004.1}) the result follows.
\end{proof}
Let us recall that by \cite[Lemma 2.4]{Zh.Ho.He} we have
\begin{align}\label{T.10004.1}
\omega(x \otimes y) = \frac{|\langle x, y\rangle| + \|x\|\,\|y\|}{2},
\end{align}
for all $x, y \in\mathcal{H}$.
Here, $x \otimes y$ denotes the rank one operator in $\mathbb{B}(\mathcal{H})$ defined by
$(x \otimes y)z := \langle z, y\rangle x$ for all $z \in\mathcal{H}$.
Notice that $(x \otimes y)^* = y \otimes x$, $r(x\otimes y) = |\langle x, y\rangle|$ and if $T\in\mathbb{B}(\mathcal{H})$,
then $T(x \otimes y) = Tx \otimes y$.

In the following theorem, we show that the numerical radius of the sequence of
the iterated $\lambda$-mean transform of a rank one operator
converges to its spectral radius.
\begin{theorem}\label{T.10004}
Let $x, y\in \mathcal{H}$ and let $\lambda \in (0, 1)$.
Then
\begin{align*}
\displaystyle{\lim_{n\rightarrow +\infty}} w\big(M^{(n)}_{\lambda}(x\otimes y)\big) = r(x\otimes y).
\end{align*}
\end{theorem}
\begin{proof}
We may and shall assume that $x, y \neq 0$.
Put $T = x\otimes y$ and $U = \frac{1}{\|x\|\|y\|}x\otimes y$.
Clearly, $\mathcal{N}(U) =\mathcal{N}(T)$ and
\begin{align*}
U^*U = \big(\frac{1}{\|x\|\|y\|}y\otimes x\big)\big(\frac{1}{\|x\|\|y\|}x\otimes y\big) = \frac{1}{\|y\|^2}y\otimes y
\end{align*}
is an orthogonal projection. Further,
\begin{align*}
|T|^2 = T^*T = \big(y\otimes x\big)\big(x\otimes y\big) = \|x\|^2y\otimes y
= \left(\frac{\|x\|}{\|y\|}y\otimes y\right)^2,
\end{align*}
and so $|T| = \frac{\|x\|}{\|y\|}y\otimes y$.
Since
\begin{align*}
U|T| = \left(\frac{1}{\|x\|\|y\|}x\otimes y\right)\left(\frac{\|x\|}{\|y\|}y\otimes y\right) = x\otimes y,
\end{align*}
from this we conclude that $U|T|$ is the canonical polar decomposition of $x\otimes y$.
Therefore,
\begin{align*}
(x\otimes y)^D = |T|U = \left(\frac{\|x\|}{\|y\|}y\otimes y\right)\left(\frac{1}{\|x\|\|y\|}x\otimes y\right)
= \frac{\langle x, y\rangle}{\|y\|^2}y\otimes y.
\end{align*}
Hence
\begin{align*}
M_{\lambda}(x\otimes y) = \lambda x\otimes y + (1-\lambda)\frac{\langle x, y\rangle}{\|y\|^2}y\otimes y
= \left(\lambda x + \frac{(1-\lambda)\langle x, y\rangle}{\|y\|^2}y\right)\otimes y.
\end{align*}
Now, by induction on $n$ we conclude that
\begin{align*}
M^{(n)}_{\lambda}(x\otimes y) = \left(\lambda^n x + \frac{(1-\lambda^n)\langle x, y\rangle}{\|y\|^2}y\right)\otimes y.
\end{align*}
Thus, by (\ref{T.10004.1}), we obtain
\begin{align*}
\omega\big(M^{(n)}_{\lambda}(x\otimes y)\big)
= \frac{\left|\left\langle \lambda^n x + \frac{(1-\lambda^n)\langle x, y\rangle}{\|y\|^2}y, y\right\rangle\right|
+ \left\|\lambda^n x + \frac{(1-\lambda^n)\langle x, y\rangle}{\|y\|^2}y\right\|\|y\|}{2},
\end{align*}
which implies
\begin{align*}
\omega\big(M^{(n)}_{\lambda}(x\otimes y)\big)
= \frac{|\langle x, y\rangle| + \sqrt{\lambda^{2n} \|x\|^2\|y\|^2 + (1-\lambda^{2n})|\langle x, y\rangle|^2}}{2}.
\end{align*}
From this it follows that
$\displaystyle{\lim_{n\rightarrow +\infty}} w\big(M^{(n)}_{\lambda}(x\otimes y)\big) = |\langle x, y\rangle|$
and the proof is completed.
\end{proof}
\begin{remark}\label{R.11004}
Let $T\in\mathbb{B}(\mathcal{H})$ and let $\lambda \in (0, 1)$.
Utilizing a similar argument as in Theorem \ref{T.8003}, we can see that
$\omega\big(M_{\lambda}(T)\big) = \omega(T)$ if and only if there exists
a sequence of unit vectors $\{x_n\}$ in $\mathcal{H}$ such that
$\displaystyle{\lim_{n\rightarrow +\infty}}\langle x_n, Tx_n\rangle\,\langle T^Dx_n, x_n\rangle = \omega^2(T)$.
\end{remark}
\bibliographystyle{amsplain}

\end{document}